\documentclass{article}
\usepackage{amsmath,amsfonts,amsthm,amssymb,mathrsfs,latexsym,paralist, xypic}
\usepackage{enumitem} 
\usepackage{upref} 
\usepackage{footnote}
\usepackage{endnotes} 
\usepackage{stmaryrd} 
\usepackage{tikz-cd} 
\usepackage{tikz}
	\usetikzlibrary{arrows,automata}
	\tikzstyle{V}=[fill=black,circle,scale=0.4, outer sep = 4pt]
\usepackage{thm-restate} 
\usepackage{tensor} 

\newtheorem{thm}{Theorem}[section]
\newtheorem{prop}[thm]{Proposition}
\newtheorem{cor}[thm]{Corollary}
\newtheorem{lemma}[thm]{Lemma}

\theoremstyle{remark}
\newtheorem{rmk}[thm]{Remark}

\newtheorem{example}[thm]{Example}
\newtheorem{examples}[thm]{Examples}
\theoremstyle{definition}
\newtheorem{defn}{Definition}

\DeclareMathOperator{\supp}{supp}

\newcommand{\z}{^{(0)}}
\renewcommand{\2}{^{(2)}}
\newcommand{\inv}{^{-1}}

\newcommand{\bi}{\begin{itemize}}
\newcommand{\ei}{\end{itemize}}
\newcommand{\be}{\begin{enumerate}}
\newcommand{\ee}{\end{enumerate}}

\renewcommand{\S}{\mathcal{S}}
\newcommand{\T}{\mathcal{T}}

\renewcommand{\H}{\mathcal{H}}

\newcommand{\G}{\mathcal{G}}
\newcommand{\K}{\mathcal{K}}

\newcommand{\id}{\operatorname{id}}
\newcommand{\Iso}{\operatorname{Iso}}
\newcommand{\norm}[1]{\left\|#1\right\|}  
\newcommand{\abs}[1]{\lvert #1 \rvert}

\date{\today}
\newcommand{\under}[2]{\ensuremath{%
#2/#1}}

\newcommand\blfootnote[1]{%
  \begingroup
  \renewcommand\thefootnote{}\footnote{#1}%
  \addtocounter{footnote}{-1}%
  \endgroup
}

\begin{document}

\title{Cartan subalgebras for non-principal twisted groupoid $C^*$-algebras}

\author{A.~Duwenig\footnotemark[\value{footnote}]\endnotemark[1], 
E.~Gillaspy\footnotemark[\value{footnote}]\endnotemark[2],
R.~Norton\footnotemark[\value{footnote}]\endnotemark[3],
S.~Reznikoff\footnotemark[\value{footnote}]\endnotemark[4],
and S.~Wright\footnotemark[\value{footnote}]\endnotemark[3]}

\endnotetext[1]{University of Victoria}
\endnotetext[2]{University of Montana}
\endnotetext[3]{Fitchburg State University}
\endnotetext[4]{Kansas State University}

\maketitle

\blfootnote{This research was begun during the 2018 BIRS workshop ``Women in Operator Algebras,'' which was partially supported by the AWM's ADVANCE grant, and continued while the authors were in residence at the Mathematical Sciences Research Institute in Berkeley, California, during the summer of 2019.  Funding for the latter stay was provided by NSA grant H98230-19-1-0119, NSF grant 1440140, the Lyda Hill Foundation, the McGovern Foundation, and Microsoft Research. The second author was also partially supported by NSF grant DMS-1600749, and the fourth author was partially supported by the Simons Foundation grant 360563. }

\begin{abstract}
{Renault proved in 2008 \cite[Theorem 5.2]{renault-cartan}} that if $\G$ is a topologically principal groupoid, then $C_0(\G\z)$ is a Cartan subalgebra in $C^*_r(\G, \Sigma)$ for any twist $\Sigma$ over $\G$.  However, there are many groupoids which are not topologically principal, yet their (twisted) $C^*$-algebras admit Cartan subalgebras.  This paper gives a dynamical description of {a class of} such Cartan subalgebras, by identifying conditions on a 2-cocycle $c$ on $\G$ and a subgroupoid $\S \subseteq \G$ under which $C^*_r(\S, c)$ is Cartan in $C^*_r(\G, c)$.   When $\G$ is a discrete group,  we also describe the Weyl groupoid and twist associated to these Cartan pairs, under mild additional hypotheses.
\end{abstract}

\section{Introduction}
A {\em Cartan subalgebra} in a $C^*$-algebra $A$ is a maximal abelian subalgebra  $B$ of $A$ satisfying certain regularity conditions (see Definition \ref{def:cartan} below).
Inspired by the work of Feldman and Moore \cite{feldman-moore} on  Cartan subalgebras in von Neumann algebras, the theory of Cartan subalgebras in $C^*$-algebras was initiated by Renault in \cite{renault}  and subsequently developed by Kumjian \cite{c*-diagonals} and Renault \cite{renault-cartan}.  

Identifying a Cartan subalgebra in a $C^*$-algebra $A$ often facilitates a   concrete  understanding of $A$, for several reasons.  First, the existence of a Cartan subalgebra $B \subseteq A$ implies that $A$ has a dynamical model  \cite{renault-cartan}, and in many situations (e.g.~\cite{matsumoto-matui,brownlowe-carlsen-whittaker}), a $C^*$-isomorphism between Cartan pairs   is equivalent to an isomorphism of the underlying dynamics.  Second, information about~$B$ can often be extended to~$A$:  for example, \cite{BNRSW}, \cite{nagy-reznikoff-1}, and \cite{brown-nagy-reznikoff} identify situations where injectivity of a representation lifts from  a  Cartan subalgebra to the entire $C^*$-algebra.
 Third, the presence of a Cartan subalgebra $B$  often enables one to apply the machinery of Elliott's classification program to $A$ (e.g., \cite{li-renault,  li-cartan-existence}).
 In particular,  \cite{li-cartan-existence} shows that for certain $C^*$-algebras,  having a Cartan subalgebra is equivalent to satisfying the Universal Coefficient Theorem, and thus implies that $A$ is indeed classified by its Elliott invariant.
Due to these applications, among others, there has been extensive research into Cartan subalgebras in recent years.

 Renault showed  \cite[Theorem 5.2]{renault-cartan} that  every separable Cartan pair arises from a twist over a topologically principal,  second countable, locally compact  Hausdorff, \'{e}tale    groupoid, and that conversely, 
 every   reduced   $C^*$-algebra of such a groupoid has a canonical Cartan subalgebra. 
 However, many natural Cartan algebras appear in the $C^*$-algebras of groupoids that are not topologically principal.  For example, the rotation algebra $A_\theta $ can be described as the $C^*$-algebra of a topologically principal groupoid $\mathbb{T} \rtimes_\theta \mathbb{Z}$, or as a twisted group $C^*$-algebra $C^*(\mathbb{Z}^2, c_\theta)$.  From the first picture and \cite[Theorem 5.2]{renault-cartan}, it is clear that $A_\theta$ has a Cartan subalgebra, but the second description of $A_\theta$ gives no hint of this.

 Another example comes from the setting of  graph $C^*$-algebras, which have a groupoid model under mild assumptions on the graph \cite{kprr}. The cycline subalgebra of a graph $C^*$-algebra $C^*(E) \cong C^*_r(\G_E)$ (introduced in \cite{nagy-reznikoff-1} as the ``abelian core'') is always Cartan in $C^*(E)$, though the  groupoid   $\G_E$ associated to $E$ is topologically principal only if the graph $E$  satisfies Condition (L), in which case the cycline subalgebra coincides with the diagonal.     In any case, the cycline subalgebra is generated by the interior of the isotropy subgroupoid of $\G_E$;  \cite[Corollary 4.5]{BNRSW} provides conditions under which for an arbitrary locally compact Hausdorff \'{e}tale groupoid $\G$,  
 $C^*_r(\Iso(\G)^\circ)$ is Cartan in $C^*_r(\G)$. 
 
 Inspired by these examples, we set out to find a dynamical description -- i.e., a description at the level of the groupoid $\G$ -- of Cartan subalgebras inside the twisted $C^*$-algebras of groupoids that need not be topologically principal. Our main result is as follows; see Definition \ref{def:imm-cent} for the meaning of ``immediately centralizing.''
 
    \newtheorem*{tg-main}{Theorem \ref{thm:main}}
\begin{tg-main}
Let $\G$ be a second countable, locally compact Hausdorff, \'etale groupoid, and let $c$ be a 2-cocycle on $\G$.  Suppose $\S $ is maximal among abelian subgroupoids of $\Iso(\G) $ on which $c$ is symmetric.  If $\S$ is clopen,   normal, and immediately centralizing, then $C^*_r(\S, c)$ is Cartan in $C^*_r(\G, c)$. 
\end{tg-main}
If $(\G, \S, c)$ satisfies the hypotheses of  Theorem \ref{thm:main}, then by \cite[Theorem 5.2]{renault-cartan}, there exists {a unique} topologically principal groupoid $\H$ -- called the \emph{Weyl groupoid} -- and a twist $\Sigma$ over $\H$ such that the Cartan pair  $(C^*_r(\G, c), C^*_r(\S, c))$ is isomorphic to the pair $(C^*_r(\H, \Sigma), C_{0}(\H\z))$.  One is then led to ask about the relationship between the original groupoid $\G$ and the new data $(\H, \Sigma)$.  We show in Proposition~\ref{lem:c_trivial_H} that if $G$ is a countable discrete group and the hypotheses of Theorem \ref{thm:main} are satisfied by $(G, S \unlhd G, c)$, 
{mild additional hypotheses guarantee that}  the Weyl groupoid $\H $ is {easily constructed from $G$.  To be precise, $\H$ is} a transformation groupoid 
$(\under{S}{G})\ltimes  \widehat{S}$. 
{We emphasize that a given group $G$ may give rise, via different subgroups $S,$ to a variety of different groupoids $\H=(\under{S}{G})\ltimes  \widehat{S}$ 
--  the Weyl groupoid is an invariant of the Cartan pair $(A, B)$, not of the enveloping $C^*$-algebra $A$. Indeed, in Section \ref{sec:apps} we exhibit a group $G$ which admits subgroups $S_1, S_2$ satisfying the hypotheses of Proposition \ref{lem:c_trivial_H} with $S_1 \cong S_2,$ but for which the associated Weyl groupoids $\H_1, \H_2$ are not isomorphic.}

Finally, we comment on the relationship between our description of the Weyl groupoid and twist in Section \ref{sec:anna} and certain results in  the recent preprint \cite{IKRSW}.  Theorem 3.4 of \cite{IKRSW} describes an (untwisted) groupoid $C^*$-algebra $C^*(\Sigma)$ as the twisted $C^*$-algebra of a quotient groupoid $\widehat{\mathcal A} \rtimes \Sigma/\mathcal A$, using a closed normal subgroupoid $\mathcal A$ of $\Iso(\Sigma)$.   Theorem \ref{thm:general-untwist-case} and Proposition \ref{lem:c_trivial_H} of the current paper offer  similar descriptions of $C^*_r(G, c)$ for a discrete group $G$ and a normal subgroup $S$. These results do not overlap with those of \cite{IKRSW}, however, because of the assumption in \cite{IKRSW} that $\G/\mathcal A$ be  topologically principal, which is never true for a nontrivial group $G/S.$  However, the consonance between their results and ours is encouraging and suggests that a unified description of the Weyl groupoids  of a larger class of algebras may be within reach.

This paper is structured as follows.  In Section~\ref{sec:background}, we recall the definitions of Cartan subalgebras, groupoids, and the $C^*$-algebras associated to groupoids.   Section~\ref{main}  is devoted to the proof of Theorem \ref{thm:main}.   Preparatory to our analysis of the Weyl groupoid associated to the Cartan pairs identified in Theorem \ref{thm:main}, Section~\ref{sec:weyl} reviews the construction in \cite{renault-cartan} of the Weyl groupoid and provides several technical results leading to a useful characterization of the elements of this groupoid. {We anticipate that Proposition \ref{prop:weyl-picture} in particular may be of independent interest.} The next section, Section \ref{sec:anna}, contains our analysis of the Weyl groupoid {and twist} arising from Theorem \ref{thm:main}, under a few additional hypotheses. We apply these results to several examples (arising from countable discrete groups) in Section~\ref{sec:apps}.   The last section, Section \ref{ex:masa-counterexample}, presents an example which showcases the necessity of the ``immediately centralizing'' hypothesis in Theorem \ref{thm:main}.  

\section{Background}
\label{sec:background}
\begin{defn}\cite[Definition 5.1]{renault-cartan}
\label{def:cartan}
Let $A$ be a $C^*$-algebra.  A $C^*$-subalgebra $B$ of $A$ is a {\em Cartan subalgebra} if:
\begin{enumerate}[label=(\arabic*)]
\item\label{item:def-Cartan:masa} $B$ is a maximal abelian subalgebra (masa) of $A$.
\item\label{item:def-Cartan:CondExp} There exists a faithful conditional expectation $\Phi\colon  A \to B$.
\item\label{item:def-Cartan:normalizer} $B$ is regular; i.e., the normalizer of $B$, 
\[ N(B) := \{ n \in A: nbn^*, n^*bn \in B \ \forall \ b \in B\},\]
generates $A$ as a $C^*$-algebra.
\item\label{item:def-Cartan:Approx1} $B$ contains an approximate identity for $A$.
\end{enumerate}
\end{defn}

Recall (cf.~\cite[Theorem 1.5.10]{brown-ozawa}) that if $B \subseteq A$ is a $C^*$-subalgebra, a map $\Phi\colon A \to B$ is a {\em conditional expectation} if $\Phi$ satisfies $\Phi_{\vert B} = \mathrm{id}_{B}$ and is contractive and  linear.
We say $\Phi$ is {\em faithful} if, for any $a \in A$, $\Phi(a^*a) = 0$ implies $a = 0$. 

A \emph{groupoid} is a generalization of a group which has inverses but only a partially defined multiplication.  Precisely, a groupoid $\G$ is a small category in which every morphism $g$ has an inverse $g^{-1}$; {we then have source and range maps  $s(g) := g^{-1} g, \ r(g) := g g^{-1}$ satisfying $r(g) g = g = g s(g)$ for all $g \in \G$.}  
The space of objects (or units) in $\G$ is  $\G\z := \{r(g) \, | \, g \in \G\}=\{s(g) \, | \, g \in \G\}.$   We denote by $\G\2$ the set $\{(g,h) \, | \, s(g)=r(h)\}$ of composable elements.  Note that since $s(gh)=s(h)$ and $r(gh)=r(g)$ whenever $g, h \in \G\2$, we have
\[ (g, h), (h, k) \in \G\2  \Rightarrow (gh, k), (g, hk) \in \G\2 \text{ and } (gh) \cdot k = g \cdot (hk).\]

The arrows-only picture of category theory allows us to identify each object $u$ with the identity morphism $\id_u \colon  u \to u \in \G$.

For each $u \in \G\z,$ we write 
\[ \G^u =  \{\gamma \in 
\G\colon  r(\gamma) = u\} \quad \quad \G_u = \{ \gamma \in \G\colon  s(\gamma) = u\}, \quad \text{ and } \quad \G^u_u=\G^u \cap \G_u.\]

The {\em isotropy subgroupoid} $\Iso(\G)$ is $\{ g \in \G:  s(g) = r(g)\}$. When the only elements in $\G$ with $s(g) = r(g)$ are the units themselves, we say $\G$ is {\em principal}.

In general, a subset $\S$ of $\G$ will be called a {\em subgroupoid} if  whenever $(s,t) \in \G\2 \cap  \S \times \S $, we have $st, t\inv , s\inv  \in \S$.  We note that in the sequel our subgroupoids will necessarily contain $\G\z$ (see Remark~\ref{rmk:units-in-subgroupoid}), so this may be taken as part of the definition. A subgroupoid $\S$ is {\em normal} if $g \S g\inv  \subseteq \S$ for all $g \in \G$, where $g \S g\inv  = \{ gtg\inv :  t\in \S, s(t) = r(t) = s(g)\}$.

 A subset $\S$ of $\G$ is {\em abelian} if, whenever $(g, h) \in \G\2 \cap \S \times \S$, we also have $(h,g) \in \G\2 \cap \S \times \S$ and $gh = hg$.

In the following, we will only consider {\em topological} groupoids; that is, groupoids $\G$  equipped with a topology such that the multiplication, range, and source maps are continuous.  In this setting, we say $\G$ is {\em topologically principal} if the set $\{ u \in \G^{(0)}: s(g) = r(g) = u \Rightarrow g = u\}$ is dense in $\G^{(0)}$.  It is worth noting that when the groupoid $\G$ is actually a group, the unit space $\G^{(0)} = \{e\}$, and thus if the group is nontrivial, then $\G$ is not topologically principal.  Nor is the path groupoid \cite{kprr} 
of a graph  that does not satisfy Condition~(L) (every cycle has an entry).   Indeed, any cycle in a graph gives rise to nontrivial isotropy elements in this groupoid; if the cycle has no entry, these elements form singleton open sets.

A topological groupoid is said to be \emph{\'{e}tale} if $r$ -- and thus also $s$ -- is a local homeomorphism. 
It is straightforward to check that in an \'etale groupoid $\G$,  the unit space $\G\z$ is clopen.  A {\em bisection} in a groupoid $\G$ is a set $B \subseteq \G$ such that there exists an open set $U\supseteq B$ for which $r\colon U \to r(U), s\colon U \to s(U)$ are homeomorphisms. If $\G$ is \'etale, the open bisections generate the topology on $\G$ \cite[Lemma 2.4.9]{AidansNotes}. 

A ($\mathbb{T}$-valued) {\em 2-cocycle} on $\G$ is a function $c\colon  \G\2 \to \mathbb{T}$ such that 
\begin{gather*}
    c(g, s(g)) = c(r(g), g) = 1 \text{ for all $g \in \G$, and}
    \\
    c(g,hk) c(h,k) = c(gh, k) c(g,h) \text{ whenever $(g,h),(h, k) \in \G\2$.}
\end{gather*} 
\begin{lemma}\label{lem:c-g-ginv}
		For $\G$ a groupoid with a 2-cocycle $c$ and $g$ any element in $\G$, we have
		$
			c(g,g\inv)=c(g\inv,g).
		$
	\end{lemma}

	\begin{proof}
		The cocycle condition gives
		\[
			1
			= c(g, g\inv g)
			= c(gg\inv, g)\, c(g,g\inv) \,\overline{c(g\inv, g)}
			= c(g,g\inv) \,\overline{c(g\inv, g)}.\qedhere
		\]
	\end{proof}

Given a continuous $\mathbb{T}$-valued 2-cocycle on $\G$, the associated full and reduced $C^*$-algebras $C^*(\G, c), C^*_r(\G, c)$ were introduced in \cite{renault}. As we will mostly be concerned with $C^*_r(\G, c)$ in this paper, we focus our discussion on this setting.

Both $C^*(\G, c)$ and $C^*_r(\G, c)$ are completions of $C_c(\G)$, which we consider as a $*$-algebra via 
\[ f* h (\gamma) = \sum_{\eta \in \G^{s(\gamma)}} f(\gamma \eta) h(\eta\inv ) c(\gamma \eta, \eta\inv ) \quad \text{and} \quad f^*(\gamma) = \overline{f(\gamma\inv )c(\gamma, \gamma\inv )}.\]
The {\em reduced norm} of $f \in C_c(\G)$ is given by representing $C_c(\G)$ on the Hilbert spaces $\ell^2(\G_u)$ for $u \in \G\z$.  To be precise, for $\xi \in \ell^2(\G_u)$ and $f \in C_c(\G)$, define $f * \xi \in \ell^2(\G_u)$ by
\[ f * \xi (\gamma) = \sum_{\eta \in \G^u} f(\gamma \eta) \xi(\eta\inv ) c(\gamma \eta, \eta\inv ).\]
Then 
\begin{equation} 
\label{eq:reduced-norm}
\| f\|_r := \sup \{ \| f * \xi\|: u \in \G\z, \, \xi \in \ell^2(\G_u), \, \| \xi \| = 1\}.
\end{equation}
The {\em reduced twisted groupoid $C^*$-algebra} $C^*_r(\G, c)$ is the completion of $C_c(\G)$ in the norm $\| \cdot \|_r$.

A continuous 2-cocycle $c$ on $\G$ gives rise to a {\em twist} over $\G$ -- that is, a groupoid $\Sigma$ with $\Sigma\z = \G\z$ and an action of $\mathbb{T}$ on $\Sigma$ that fixes the unit space, such that $\Sigma/\mathbb{T} \cong \G.$  In particular, given $c$, we take $\Sigma = \G \times \mathbb{T}$ as topological spaces, and define the multiplication in $\Sigma$ by 
\[ (\gamma, t)(\eta, s) = (\gamma \eta, c(\gamma, \eta) ts)\]
whenever $(\gamma, \eta) \in \G\2.$  

The full and reduced $C^*$-algebras $C^*(\G, \Sigma), C^*_r(\G, \Sigma)$ of a twist $\Sigma$ over $\G$ are defined as completions of 
\[\{ f \in C_c(\Sigma): f(z \gamma) = \overline{z}f(\gamma) \text{ for all } z \in \mathbb{T}, \gamma \in \Sigma\};\]
see \cite[Example 2.9]{equiv-disint} for the details. If $\Sigma$ arises from a 2-cocycle $c$, then one can compute that $C^*(\G, \Sigma) \cong C^*(\G, c)$ and $C^*_r(\G, \Sigma) \cong C^*_r(\G, c).$

  As explained in (\cite[1.6]{c*-diagonals}, \cite[Proposition 4.7]{renault-cartan}), a separable Cartan pair $(A,B)$ gives rise to a topologically principal groupoid and twist as follows.   For 
    any $n\in N(B)$, there exists a unique partial homeomorphism $\alpha_n$ with domain
    \[
        \mathrm{dom}(n)
        :=
        \left\{
            x\in \widehat{B}
            \,\middle|\,
            n^*n(x) >0
        \right\},
    \]
    where $\widehat{B}$ is the Gelfand dual of $B$, and with codomain $\mathrm{dom}(n^*)$, that satisfies
    \begin{align}\label{eq:defining-eq-for-alpha} \forall \  b \in B \quad
        n^* b n 
        =
        (b\circ\alpha_{n}) \cdot n^* n.
    \end{align} If $n,m\in N(B)$, then one can show that
    \[
        \mathrm{dom}(nm)
        =
        \left\{
            x\in \mathrm{dom}(m) \,\middle|\,
            \alpha_m (x) \in \mathrm{dom}(n)
        \right\},
    \]
and on this domain, $\alpha_n \circ \alpha_m$ and $\alpha_{nm}$ agree. Furthermore, $\alpha_{n^*} = \alpha_n\inv.$
The family $\{\alpha_n\}_{n \in N(B)}$ gives rise to the {\em Weyl groupoid}  $\G_{(A,B)}$ of the Cartan pair: As a set, $\G_{(A,B)}$ is the quotient
\[
    \left\{
        (\alpha_{n}(x),n, x) \mid
        n\in N(B),
         x\in \mathrm{dom}(n)
    \right\} / \sim
\]
under the equivalence relation \begin{align}
\begin{split}\label{def:sim}
   & (\alpha_n(x), n, x) \sim
    (\alpha_m( y), m, y)
     \iff   \\
   &x=y  
    \text{ and there exists an open $U\subseteq \widehat{B}$ with }
   x\in U \text{ and }
    \alpha_n\vert_{U}
    = \alpha_m\vert_{U}
    .
\end{split}
\end{align}
We shall  denote the equivalence class of $(\alpha_n( x), n, x)$  by $[\alpha_n(x), n, x]$. It can be verified (cf.~\cite{renault-cartan}, \cite[Prop.~5.1.15]{AidansNotes}) that the  composition given by 
\[
    [\alpha_n(\alpha_m(x)), n, \alpha_m(x)]
    \cdot
    [\alpha_m(x), m, x]
    =
    [\alpha_{nm}(x), nm, x]
\]
is well-defined. 
We define a basic open set in $\G_{(A, B)}$ to be of the form $\{ [\alpha_n(x), n, x]: \alpha_n(x) \in V, x \in U\}$ for $U, V \subseteq \widehat B$ open and $n \in N(B)$ (cf.~\cite[Section 3]{renault-cartan}).

Similarly to the Weyl groupoid, the {\em Weyl twist} $ \Sigma_{(A,B)}$ is a quotient of
\[
    \left\{
        (\alpha_{n}(x), n ,x) \, \vert \,
        n \in N(B), 
        x \in \mathrm{dom}(n) 
    \right\}
    ,
\]
but under the following, more rigid, equivalence relation:
\begin{align*}
    &(\alpha_n(x), n, x) \approx
    (\alpha_m(y), m, y) \\ 
    \iff  &
    x = y
    \text{ and }
    \exists \ b,b'\in B \text{ such that }
    b(x),b'(x) >0 \text{ and }
    nb=mb'
    .
\end{align*}
We write $\llbracket \alpha_n(x), n, x \rrbracket$ for the class of the triple $(\alpha_n(x), n, x)$ with respect to this equivalence relation. Notice that equivalence with respect to $\approx$ implies equivalence with respect to $\sim$.
Thus, setting  (for $\lambda \in \mathbb{T}$)  
$$\lambda \cdot \llbracket \alpha_n(x), n, x \rrbracket = \llbracket \alpha_n(x), \lambda n, x \rrbracket$$
gives a well-defined action of $\mathbb{T}$ on $\Sigma_{(A,B)}$, and one can check (cf.~\cite[Proposition 4.14]{renault-cartan}) that $\Sigma_{(A, B)}/\mathbb{T} \cong \G_{(A,B)}$.

\section{Main result} \label{main}

In this section we prove Theorem \ref{thm:main}, which identifies Cartan subalgebras inside the $C^*$-algebras of twisted groupoid $C^*$-algebras that need not be topologically principal.

\begin{thm}
\label{thm:main}
Let $\G$ be a second countable, locally compact Hausdorff, \'etale groupoid, and let $c$ be a 2-cocycle on $\G$.  Suppose $\S $ is maximal among abelian subgroupoids of $\Iso(\G) $ on which $c$ is symmetric.  If $\S$ is clopen,   normal, and immediately centralizing, then $C^*_r(\S, c)$ is Cartan in $C^*_r(\G, c)$. 
\end{thm}

We begin with a discussion of the ``immediately centralizing'' hypothesis needed for Theorem \ref{thm:main}, and then establish each of the four properties of Cartan subalgebras in a separate proposition.

\begin{defn}\label{assu:A}
\label{def:imm-cent}
Given a subgroupoid $\S \subseteq \G$, we say an element $\nu \in  \G^u_u \subseteq \Iso(\G)$ is  \emph{$k$-centralizing} for $\S$   (for $k \geq 1$) if for all $t \in \S^{u}_{u}$ there exists  $j \in \{1, \ldots, k\}$ such that $\nu t^j=t^j \nu$.      We will be concerned with subgroupoids $\S$ of $\Iso(\G)$ where all $k$-centralizing elements are in fact $1$-centralizing. 

That is, letting $C_k(\S)=\{\nu \in \G \, | \, \nu \text{ is $k$-centralizing for }\S\}$, we say that $\S$ is   \emph{immediately centralizing} if $\forall \  k \geq 1$, $C_k(\S) = C_1(\S)$.
 \end{defn}

    The property might seem very technical, so let us give two quick examples.  We thank Caleb Eckhardt for introducing us to the unique root property.

\begin{examples}\label{rmk:Examples-for-(A)} \hfill

   \begin{enumerate}
       \item When $\Iso(\G)$ is abelian, any subgroupoid of it is immediately centralizing.  This is the case for the description of the irrational rotation algebra from Example \ref{ex:IRR}, as well as for the path groupoid of any $k$-graph (see Section~2 of \cite{kp}).
         \item A subgroupoid $\S$ is immediately centralizing if each isotropy group $\S^u_u$ has the \emph{unique root property}: if $s,t\in \S^u_u$ are such that $s^{j} = t^{j}$ for some $j\in\mathbb{N}$, then $s=t$. (See \cite{Baumslag-roots} for a treatment of groups with the unique root property). In this situation, the equation $\nu t^j = t^j \nu$, or in other words $(\nu t \nu\inv)^j = t^j$, implies $\nu t \nu\inv = t$. We will study an example of such a group  in Section \ref{sec:apps}.
    \end{enumerate}
\end{examples}

For the rest of this section, $\G$ will always denote an {\'e}tale groupoid, $c$ a 2-cocyle on $\G$, and $\S$  a subgroupoid of $\Iso (\G)$.

\begin{rmk} \label{rmk:units-in-subgroupoid} Note that $\G\z$ is an abelian subgroupoid of $\Iso(\G)$ on which $c$ is symmetric.  Indeed, given any abelian subgroupoid $\mathcal A$ of $\Iso(\G)$ on which $c$ is symmetric, the set $\mathcal A \cup \G\z$ is another such.
It follows that any subgroupoid $\S$ satisfying the hypotheses of Theorem \ref{thm:main} will contain $\G\z$. 
\end{rmk}

\begin{rmk} 
\label{rmk:inclusion}
As discussed in the introduction to \cite[Section 3]{BNRSW}, when $\S$ is an open subgroupoid of $\Iso(\G)$, \cite[Proposition 1.9]{phillips-xprods} tells us that the map $\iota\colon  C_c(\S) \to C_c(\G)$ given by \[\iota(f)(g) =\begin{cases} 0, & g\not\in S \\
f(g), & g \in S 
\end{cases}\]
extends to an injective $*$-homomorphism 
from $C^*_r(\S)$  into $C^*_r(\G).$ A careful examination of Phillips' proof (using $\displaystyle u(\xi) = \gamma \mapsto \xi(\gamma g_0) \overline{c(\gamma, g_0)}$) reveals that $\iota$ also induces an injective $*$-homomorphism of $C^*_r(\S, c)$ into $C^*_r(\G,c).$ Because of this, we will make no notational distinction between $C^*_r(\S, c)$ and $\iota(C^*_r(\S, c)) \subseteq C^*_r(\G, c).$
\end{rmk}

	\begin{lemma}
	If $\S$ is abelian and $c$ is symmetric on $\S$, then $C_r^* (\S,c)$ is abelian.
	\end{lemma}

	\begin{proof}
		It suffices to check that $C_c(\S,c)$ is abelian. We compute for $u\in \G\z = \S\z$ and $s\in\S^{u}_{u}$,
		\begin{alignat*}{3}
			(f\ast g) (s) 
			&=
			\sum_{t\in \S^{u}_{u}}
				f(st) g(t\inv) \, c(st, t\inv)
			\\
		&	=
			\sum_{r\in \S^{u}_{u}}
				f(r\inv) g(sr) \, c(r\inv, sr)
			&\quad
			\text{($rs = sr$)}
			\\
		&	=
			\sum_{r\in \S^{u}_{u}}
				 g(sr) f(r\inv) \, c(sr,r\inv)
			&\quad\text{($c$ symmetric on $\S$)}
			\\
			&= (g\ast f) (s).
		\end{alignat*}
		\par \vspace{-1.7\baselineskip} 
    \qedhere
	\end{proof}

One might be tempted to think that the following result follows immediately from the definitions.  However, this is emphatically not the case.
	
	\begin{lemma}\label{lem:eta-in-S}
	    Suppose  $\S$ is maximal among abelian subgroupoids of $\Iso (\G)$ on which $c$ is symmetric. Let $u$ be a unit. If $\eta\in\G^{u}_{u}$ satisfies $\eta s = s \eta$ and $c(s, \eta) = c(\eta, s)$ for all $s\in \S^{u}_{u}$, then $\eta\in\S$.
	  \end{lemma}
	
	The key difficulty is that the assumption that $c(s, \eta) = c(\eta, s)$ for all $s \in \S$ does not immediately imply that $c(s, \eta^k) = c(\eta^k, s)$ for all $k\in\mathbb{Z}$, and so there is non-trivial work required to prove that $c$ is also symmetric on the subgroupoid generated by $\S$ and $\eta$ (which, by maximality, then implies $\eta \in \S$). While the proof of Lemma \ref{lem:eta-in-S} is fairly long and intricate, it consists primarily of several careful applications of induction and is not very enlightening, so we relegate it to Appendix \ref{app:lemma-4-8}.
		
	\begin{lemma}\label{lem:prop-of-h}
		Suppose that  $\G$ is a second countable, locally compact  Hausdorff,  \'{e}tale  groupoid with 2-cocycle $c$. 
		  If $h\in C_r^* (\G, c)$ commutes with every element of $C_c (\S, c)$, then $h$ is supported in $\Iso (\G)$ and satisfies
		\begin{equation}\label{eq:prop-of-h}
	    	h(\nu) c(s, \nu) 
			=
			h(s \nu s\inv) c(s \nu s^{-1}, s) 
		\end{equation}
		for all $\nu\in \Iso (\G)$ and all $s\in \S$ with the same range \textup(and source\textup) as $\nu$.
	\end{lemma}
		\begin{proof}
		From Theorem 4.2 in \cite{renault-cartan} and the discussion above it, we get $\supp (h)\subseteq \Iso (\G)$, where $h$ is thought of as an element of $C_0 (\G)$. 
		
		Fix any $\nu\in \Iso (\G)$ and any $ s \in \S$ with the same source as $\nu$ (and hence also the same range). We can find a bisection $B$ such that $\G_{s(\nu)} \cap B = \{ s \}$ and a function $f\in C_c (\S, c)$ whose support is contained in $B$ with $f( s ) = 1.$  Consequently, $s$ is the only element of $\G_{s(\nu)}$ with $f(s) \not= 0$.   Using 
			the fact that $\supp(h) \subseteq \Iso(\G)$, we have 
		\begin{align*}
		    (h \ast f)(s \nu) & = \sum_{\eta \in s^{-1}(s(\nu))}  h(s\nu\eta^{-1})f(\eta)c(s\nu\eta^{-1}, \eta) \\
		    & = h(s\nu s^{-1}) c(s\nu s^{-1}, s)
		\end{align*}
		and
		\begin{align*}
		    (f\ast h)(s \nu) & = \sum_{\zeta \in s^{-1}(s(\nu))}  f(s \nu \zeta\inv) h(\zeta) c(s \nu \zeta\inv, \zeta)\\
			& = \sum_{\zeta \in \G_{s(\nu)}^{s(\nu)}} f(s \nu \zeta^{-1}) h(\zeta) c(s \nu \zeta^{-1}, \zeta)\\
			& = f(s) h( \nu) c(s, \nu )\\
			& = h(\nu) c(s, \nu)
		\end{align*}
		Since $h$ and $f$ commute, this completes the proof.
	\end{proof}

Recall from \cite[Proposition II.4.2]{renault} that for any \'etale groupoid $\G$, we have an injective, norm-decreasing inclusion $j: C^*_r(\G,c) \to C_0(\G)$, where the latter space is equipped with the supremum norm.

    \begin{lemma}\label{lem:Bs-elts-have-supp-in-S}
        Suppose the subgroupoid $\S$ is clopen. Then an element of $C_{r}^{*}(\G,c)$ is in the subalgebra $C_{r}^{*}(\S,c)$ if and only if its image in $C_{0}(\G)$ is supported in $\S$.
    \end{lemma}
    \begin{proof}
         
        First, if $h\in C_{r}^{*}(\S, c)$, then there exist $h_{n}$ in $C_c (\S,c)$ which converge to $h$ and so $j (h_{n}) = h_{n}$ converge to $j (h)$.  The fact that each $h_n$ is supported on $\S$ therefore implies that if $y \not\in \S$,
        \[ | j(h)(y)| = | j(h)(y) - h_{n} (y)|  \leq \| j(h) - h_n\|_{\infty} \leq \| h-h_n\|_r\]
        can be forced less than $\epsilon$ for any $\epsilon > 0$.  Since $\S$ is closed, we conclude that $\text{supp}(j(h)) \subseteq S.$
        
        Conversely, assume that the image of $h\in C_{r}^{*}(\G,c)$ under $j$ is supported in $\S$. Let $h_{n} \in C_c (\G,c)$ converge to $h$ in the reduced norm, and define $h'_{n} := \chi_{\S}\cdot h_{n}$. Since $\S$ is clopen, $h'_{n}\in C_{c}(\S,c)$. As $\norm{\chi_{\S}\cdot f}_{r}\leq \norm{f}_{r}$ for any $f\in C_{c}(\G,c)$, it follows from the fact that $h_{n}{\to} h$ that  $(h'_{n})_{n}$ is Cauchy. Since $j$ is continuous, we have that $j(h_{n})$ converges to $j(h)$ in $C_{0}(\G)$. As $j(h)$ is supported in $\S$, this implies that $\chi_{\S}\cdot j(h_{n}) = \chi_{\S}\cdot h_{n} = j(h_{n}')$ also converges to $j(h)$. In particular, since $j$ is injective, we must have that the $C^*$-limit of $(h'_{n})_{n}$ coincides with the $C^*$-limit of $(h_{n})_{n}$, i.e.~$h$ is the limit of elements in $C_c (\S,c)$ and hence an element of $C_r^* (\S,c).$
    \end{proof}

    \begin{prop}\label{prop:masa}
		With all the assumptions from Theorem~\ref{thm:main}, $C^*_r (\S,c)$ is maximal abelian in $C^*_r (\G,c)$.
	\end{prop}

    \begin{rmk}
        If $\S$ is not immediately centralizing, then this statement is not {necessarily} true. See the example in Section \ref{ex:masa-counterexample}.
    \end{rmk}

	\begin{proof}
     By Lemma~\ref{lem:Bs-elts-have-supp-in-S}, $h\in C_r^* (\G,c)$ is in $ C_r^* (\S,c)$ exactly when its image in $C_{0}(\G)$ has support in $\S$. Consequently, if we assume that $h\in C_r^* (\G,c)$ commutes with every element of $C_r^* (\S,c)$, then we need to show that $\mathrm{supp}(h)\subseteq\S$ in order to conclude that $C_{r}^{*}(\S,c)$ is maximal abelian. As $\S$ is closed, it suffices to check that $h(\nu)\neq 0$ implies $\nu\in \S$. Note that we already know by Lemma \ref{lem:prop-of-h} that $\nu \in \G^u_u$ for some unit $u$.
		
		First, suppose that $\nu$ commutes with every $s\in \S^u_u$. Then Equation \eqref{eq:prop-of-h} implies
		    $c(\nu, s)
			=
			c(s, \nu)$
		for all $s\in \S^u_u$. It follows from Lemma \ref{lem:eta-in-S} that $\nu\in \S$, as desired.
		
		Next, suppose  there exists at least one $s \in \S^u_u$ such that $\nu s \not= s \nu,$ i.e.~$\nu \notin C_1(\S)$ (recall Definition~\ref{def:imm-cent}).  Since $\S$ is immediately centralizing, this implies that for all $k \geq 1$, $\nu \notin C_k(\S)$ and therefore for every $k \geq 1$, there exists a $t \in \S^u_u$ such that the set $\{t^j \nu t^{-j} \, | \, 1 \leq j \leq k\}$ has cardinality {$k$}.  It follows that  the set
	    \[
	        T
	        :=
	        \{
	            t \nu t\inv
	            \, | \, t\in \S^u_u
	        \}
	    \]
	    is infinite. 
	    However, Equation \eqref{eq:prop-of-h} implies that for any $\eta\in T$,
	    \[
	        \vert h(\eta) \vert
	        =
	        \vert h(\nu) \vert
	        > 0
	        .
	    \]
    	Now, if $K$ is any compact subset of $\G$, then the discrete set $\G^u_u$ must have finite intersection with $K$. In particular, the infinite set $T$ cannot be fully contained in $K$. So we have shown that there exists an $\epsilon >0$, namely $\epsilon := \vert h(\nu) \vert$, such that for any compact $K \subseteq \G$, there exists an $\eta\in \G$, namely $\eta\in T\cap (\G\setminus K)$, so that $\vert h(\eta) \vert \geq \epsilon$. Therefore, $h \not \in C_0(\G)$, which (by \cite[Proposition II.4.2]{renault}) contradicts our assumption that $h\in C_r^* (\G,c)$.
	\end{proof}

We require the following lemma to show that the normalizer of $C_r^* (\S,c)$ generates $C_r^*(\G,c)$ as a $C^*$-algebra. For the definition of the normalizer, see Definition~\ref{def:cartan}, Item~\ref{item:def-Cartan:normalizer}.

\begin{lemma}\label{lem:reg}
Suppose that the subgroupoid $\S$ is normal. If $h \in C_c(\G,c)$ is supported in a bisection, then $h$ is in the normalizer of $C_c(\S,c)$ in $C_c(\G, c)$. 
\end{lemma}

\begin{proof}
Suppose $f \in C_c(\S,c)$. Since $h \ast f \ast h^*$ and $h^* \ast f \ast h$ are continuous functions with compact support, we only need to show that they are supported on $\S$. For $\xi \in \G$ we have
\begin{align*}
    \left( h \ast f \right) \ast h^*(\xi)
    &= \sum_{\rho \in \G^{s(\xi)}} (h \ast f)(\xi \rho) h^*(\rho\inv ) c(\xi \rho, \rho\inv ) \\
    &= \sum_{\rho \in \G^{s(\xi)}} \left( \sum_{\nu \in \G^{s(\rho)}} h(\xi \rho \nu) f(\nu\inv ) c(\xi \rho \nu, \nu\inv ) \right)  h^*(\rho\inv ) c(\xi \rho, \rho\inv ) \\
    &= \sum_{\rho \in \G^{s(\xi)}} \sum_{\nu \in \G^{s(\rho)}}  h(\xi \rho \nu) f(\nu\inv ) \overline{h(\rho) c(\rho\inv , \rho)} c(\xi \rho \nu, \nu\inv ) c(\xi \rho, \rho\inv ).
\end{align*}

If a term in the sum is nonzero, then we must have $\nu\inv \in \S$\label{a-spot-where-we-use-S-sset-Iso}, 
and also both $\rho$ and $\xi \rho \nu$ must be in $\supp(h)$. Since $\nu\in\S \cap \G^{s(\rho)}\subseteq\Iso(\G)$, we have that $s(\xi \rho \nu) =  s(\nu) =  r(\nu) = s(\rho)$. Since $h$ is supported on a bisection, this implies $\xi \rho \nu = \rho$. Thus, the only summand that might not vanish corresponds to $\rho$ and $\nu$ such that $\xi = \rho \nu\inv \rho\inv$. The normality of $\S$ thus implies that $\xi \in \S$. Therefore $h \ast f \ast h^* \in C_c(\S,c)$.

A similar calculation shows that $h^* \ast f \ast h$ is supported on $\S$. Therefore $h$ lies in the normalizer of $C_c(\S,c)$.
\end{proof}

\begin{prop}\label{prop:normalizer-generates}
Assume the {\'e}tale groupoid $\G$ with 2-cocycle $c$ is locally compact and Hausdorff, 
and that the subgroupoid $\S$ of $\Iso (\G)$ is normal. 
Then the normalizer $N(C_r^*(\S,c))$  of $C_r^*(\S,c)$
generates $C_r^*(\G,c)$ as a $C^*$-algebra. 
\end{prop}

\begin{proof}
Suppose $h \in C_c(\G, c)$. Since $\G$ is {\'e}tale, its topology has a basis of open bisections (see \cite[Proposition 3.5]{Exel:Inverse-semigps} or \cite[Lemma 2.4.9]{AidansNotes}); in particular, we can take a finite collection $\{U_i\}_{i=1}^n$ of such sets which cover the compact support of $h$. As $\G$ is assumed to be locally compact Hausdorff, we can choose a partition of unity $\{\xi_i\}_{i=1}^n$ subordinate to that cover. 
The pointwise products $h_i=\xi_i \cdot h$ belong to $C_c(\G, c)$ with $\supp(h_i) \subseteq U_i$, and $h=\sum_{i=1}^n h_i$.

By Lemma \ref{lem:reg}, $h_i$ is in the normalizer of $C_c(\S,c)$ for each $i=1, \ldots, n$. Therefore $h=\sum_{i=1}^n h_i$ is in the normalizer of $C_c(\S,c)$. Thus $C_c(\G,c)$ is contained in the normalizer of $C_c(\S,c)$.

Now, suppose $(f_{n})_n \subseteq  C_c (\S, c)$ converges to $f$ in $C_{r}^{*}(\S, c)$. If $h\in C_{c} (\G, c)$, then $h\ast f_{n} \ast h^*$ is an element of $C_{c} (\S, c)$ by the above argument and so its  $C_{r}^* (\G, c)$-limit $h\ast f \ast h^*$ is an element of $C_{r}^* (\S, c)$. We have shown that $C_{c} (\G, c)$ is also contained in the normalizer of $C_r^*(\S,c)$, which hence generates $C_r^*(\G,c)$ as a $C^*$-algebra. 
\end{proof}
\begin{prop} \label{prop:conditional-expectation}
Assume 
the subgroupoid $\S$ of $\Iso(\G)$ is clopen. Then there is a faithful conditional expectation $\Phi\colon  C_r^*(\G,c) \rightarrow C_r^*(\S,c)$.
\end{prop}

\begin{proof} Since $\S$ is open in $\G$, there is an injective $*$-homomorphism $\iota\colon  C_c(\S,c) \rightarrow C_c(\G,c)$ given by extension by zero.   By Remark \ref{rmk:inclusion}, the function $\iota$  extends to an inclusion \[\iota\colon  C_r^*(\S,c) \rightarrow C_r^*(\G,c).\] For this proof, let $M_r=\iota(C_r^*(\S,c)) \subseteq C_r^*(\G,c).$ We claim that the function $\Phi_{0}\colon  C_c(\G, c) \rightarrow M_r$   given by $\Phi_0(f) = \iota(f|_{\S})$  extends to a conditional expectation.

First, observe that $f|_{\S} \in C_c(\S,c)$ for all $f \in C_c(\G,c)$, because $\S$ is clopen. 
Thus, $\Phi_{0}$ is well-defined. 
Clearly, $\Phi_0$ is linear and idempotent on $C_c(\S,c)$.  
We will show that $\Phi_0$ is contractive, i.e.~that for $f \in C_c(\G,c),$ we have
\[
\|\Phi_0(f) \|_{M_r} \leq \|f\|_{C_r^*(\G,c)},\]
so that $\Phi_0$ will extend to a linear, contractive map $\Phi$ on all of $C_r^*(\G, c)$ which fixes $C_{r}^{*}(\S, c)$ and is hence a conditional expectation.

{For $u \in \G^{(0)}$,} let $L^u\colon  C^*_r(\S, c) \rightarrow \mathcal{B}(\ell^2(\S_u))$ be the left regular representation, given for $g\in C_c (\S,c)$ and $\xi\in \ell^2 (\S_u)$ by
\begin{equation}
\label{eq:left-reg}
L^u(g)(\xi) = g \ast \xi =
\left[ \gamma \mapsto \sum_{\eta \in \G^u} g(\gamma \eta) \xi(\eta^{-1}) c(\gamma \eta, \eta^{-1}) \right]
.\end{equation}
By the definition of the norm $\| \cdot \|_r$ on $C_r^* (\S,c)$ (see Equation \eqref{eq:reduced-norm}),  
we can find, for any $f\in C_c (\G,c)$, a unit $u \in \S\z = \G\z$ such that  
$\|\iota(f|_{\S})\|_r  \leq \|L^u(\iota(f|_{\S}))\| +\epsilon$. 

Let $\pi_u\colon  C^*_r(\G,c) \rightarrow \mathcal{B}(\ell^2(\G_u))$ be the left regular representation (given by the same formula as Equation \eqref{eq:left-reg}).  Let $P \in \mathcal{B}(\ell^2(\G_u))$ be the orthogonal projection onto
$\overline{\mathrm{span}} \{\mathrm{e}_\gamma  :  \gamma \in \S_u\}
$, where $\{\mathrm{e}_{\gamma}\}_{\gamma\in\G_{u}}$ denotes the standard orthonormal basis of $\ell^{2}(\G_{u})$.
There is a canonical unitary isomorphism $\Gamma\colon  \ell^2(\S_u) \rightarrow P\ell^2(\G_u)$ given by $\Gamma(\sum_{\gamma \in \S_u} a_\gamma \mathrm{e}_\gamma)= \sum_{\gamma \in \G_u} \chi_{\S}(\gamma) a_\gamma \mathrm{e}_\gamma$, where $\chi_{\S}$ denotes the characteristic function of $\S$.   It is easy to check that for all $\xi \in \ell^2(\S_u)$, 
\[ P \pi_u(f)(\Gamma(\xi)) = 
    \Gamma
    \Bigl(
        L^u
            \bigl(
                \iota(f|_{\S})
            \bigr)(\xi)
    \Bigr),\]
and so with $\epsilon$ and $u$ as above for the fixed $f\in C_c(\G,c)$, we have
\[ \|\Phi_0(f) \|_r = \|\iota(f|_{\S})\|_r \leq \|L^u(\iota(f|_{\S}))\|+\epsilon = \|P\pi_u(f)\| + \epsilon \leq \|f\|_r +\epsilon.\]

Thus $\Phi_0$ extends to a linear idempotent $\Phi\colon  C_r^*(\G,c) \rightarrow M_r$, which has  norm 1 since it acts as the identity on $M_r \subseteq C^*_r(\G, c)$. By \cite[II.6.10.2]{blackadar-opalgs}, $\Phi$ is a conditional expectation.
To see that $\Phi$ is faithful, 
 we follow the same idea as in the proof of \cite[II.4.8]{renault}. For $f\in C_{r}^*(\G,c)$ and $u\in \G\z=\S\z$, we have
 \[
    \Phi (f^* \ast f)(u)
    =
    (f^* \ast f)(u)
    =
    \sum_{\gamma \in \G_{u}} f^* (\gamma\inv) f(\gamma) c(\gamma, \gamma\inv)
    =
    \sum_{\gamma \in \G_{u}} \abs{f(\gamma)}^{2}.
 \]
 In particular, if $\Phi (f^* \ast f) =0$, then $\sum_{\gamma \in \G_{u}} \abs{f(\gamma)}^{2}=0$ for each $u\in \G\z$ and hence $f(\gamma)=0$ for every $\gamma\in \G$.
\end{proof}

\begin{proof}[Proof of Theorem \ref{thm:main}]
We know from the remarks preceding \cite[Proposition 4.1]{renault-cartan} that $C_0(\G\z)$ always contains an approximate unit for $C^*_{r}(\G,c)$; hence, so does $C^*_r(\S, c)$, and Condition \ref{item:def-Cartan:Approx1} of Definition \ref{def:cartan} holds.  Propositions \ref{prop:masa}, \ref{prop:conditional-expectation}, and \ref{prop:normalizer-generates}, imply, respectively, that Conditions \ref{item:def-Cartan:masa}, \ref{item:def-Cartan:CondExp}, and \ref{item:def-Cartan:normalizer}  of Definition \ref{def:cartan} are satisfied.
\end{proof}

\section{Weyl groupoid and Weyl twist}
\label{sec:weyl}
Having identified  Cartan subalgebras inside twisted groupoid $C^*$-algebras in Theorem \ref{thm:main}, we can use Renault's machinery \cite{renault-cartan} to identify an alternative   groupoid model---one that is topologically principal---for these algebras.  Our next goal is to analyze the relationship between the original groupoid data $(\G, \S, c)$ and the Weyl groupoid and twist associated to the Cartan pair $(C^*_r(\G, c), C^*_r(\S, c))$ as in \cite[Section 4]{renault-cartan}.  Section \ref{sec:anna}  analyzes the general structure of this relationship in the setting when $\G$ is a discrete group, and Section \ref{sec:apps} computes the Weyl groupoid and twist explicitly in several examples.

The results in this section, particularly Proposition \ref{prop:weyl-picture}  which is the main result of this section, will facilitate our description and analysis of the Weyl groupoid in Section \ref{sec:anna}.  We heartily thank Aidan Sims for  suggesting Proposition \ref{prop:weyl-picture} to us, and for helpful discussions relating to its proof.

Throughout the current section,  $B$ will denote a Cartan subalgebra of a  separable $C^*$-algebra $A$ with $\Phi\colon A\to B$ the conditional expectation.

 \begin{prop}
\label{prop:weyl-picture}
    Suppose there exists a subset $N$ of $N(B)$ which densely spans $A$. Then every element of the Weyl groupoid associated to $(A,B)$ can be represented by some $(\alpha_m(x), m, x)$ where $m \in N$.
\end{prop}
Before proving this result we will make some clarifying observations about the Weyl groupoid defined in Section \ref{sec:background} and prove several preparatory lemmas.  Recall Equation~\eqref{eq:defining-eq-for-alpha}, the    defining equation of the partial homeomorphisms $\alpha_n$:   $$\forall \ b \in B , \quad
        n^* b n 
        =
        (b\circ\alpha_{n}) \cdot n^* n.$$ Notice that $b \circ \alpha_{n}$ is a function that might  be only partially defined \textup(in which case it is not an element of $B$\textup). But since the function $n^*n \in B \cong C_0(\widehat B)$ vanishes wherever $ \alpha_{n}$  does not make sense, one unambiguously defines for any $x\in \widehat{B}:$
\begin{align*}
        \bigl( n^* n \cdot  (b\circ \alpha_{n }) \bigr) (x)
        =  & \, 
        \bigl( (b\circ \alpha_{n }) \cdot n^* n   \bigr) (x)
        \\
        :=&
        \left\{
            \begin{array}{cc}
                 n^* n (x) \cdot  b (\alpha_{n } (x))& \text{ if } x\in \mathrm{dom}(n),  \\
                0 & \text{ otherwise.} 
            \end{array}
        \right.
    \end{align*}

    Indeed, $n^*n \cdot (b \circ \alpha_n) \in C_0(\widehat B) \cong B$ for any $b \in B, n \in N(B),$ since $b, n^*n \in C_0(\widehat B)$ and $\alpha_n$ is a homeomorphism defined on the domain of $n^*n.$

Proofs of the following two lemmas are straightforward using Equation \eqref{eq:defining-eq-for-alpha} (which uniquely determines $\alpha_{n}$), the $C^*$-identity, and the fact that $B$ is maximal abelian in $A$. 

\begin{lemma}\label{lem:stronger-def-for-alpha}
    Suppose $b\in B$  vanishes outside of $\mathrm{dom}(n)$, so that 
    \[
        (b\circ\alpha_{n^*}) (x)
        =
        \begin{cases}
            0 & \text{ if } x\notin \mathrm{dom}(n^*),
            \\
            b \left( \alpha_{n^*} (x) \right)
            & \text{ otherwise},
        \end{cases}
    \]
    is a globally defined continuous function on $\widehat{B}$, i.e.\ an element of $B$. Then $nb = (b\circ\alpha_{n^*})n$.
\end{lemma}

\begin{lemma}\label{lem:alpha-n=id-implies-n-in-B}
    If $n\in N(B)$ has the property that $\alpha_n = \mathrm{id}|_{\mathrm{dom}(n)}$, then $n\in B$.
\end{lemma}

\begin{lemma}  \label{lem:nbhd for alpha_i}
    Let $n, m \in N(B)$. If either
    \begin{enumerate}[label=(\arabic*), font=\upshape]
        \item\label{item:not-in-domain=Phi-zero} $x\notin\mathrm{dom}(n) \cap \mathrm{dom}(m)$, or
        \item\label{item:in-domain+alph-uneq=Phi-zero} $x \in \mathrm{dom}(n) \cap \mathrm{dom}(m)$ satisfies $\alpha_n(x) \neq \alpha_m(x)$, 
    \end{enumerate}
    then $\Phi(n^*m)(x)=0$.
\end{lemma}
    
\begin{proof} 
First assume that $x$ is not in the domain of $m$, say, so that $m^* m (x) = 0$. 
Fix $\epsilon>0.$ Continuity of $m^*m$ implies that there exists a neighborhood $U$ of $x$ such that 
\[
    \sup \{ m^*m(y) : y \in U\} < \epsilon.
\]
Let $b\in B$ be a $[0,1]$-valued function such that $b(x)=1$ and $b$ vanishes off of $U$. Then
\begin{align*} 
    \norm{m b}^2
    &
    = 
    \norm{(mb)^* m b}
    =
    \sup\{ m^*m (y) b^2(y): y \in \widehat B\}\\
    &\leq
    \sup \{ m^*m(y): y \in \widehat B \text{ such that } b(y)\neq 0\}\\
    &\leq
    \sup \{ m^*m(y): y\in U\}
    <
    \epsilon
    .
\end{align*}
Since $b(x) =1$, the $B$-linearity of the conditional expectation $\Phi$, and the fact that $\Phi$ is norm-decreasing, now imply that 
\begin{align*}
    \Phi(n^*m)(x)&  = \Phi(n^*m)(x) b(x)= \Phi(n^*m b)(x) \leq \norm{ \Phi(n^*m b)} \leq\norm{n^* m b} \\
    & \leq \norm{n} \norm{m b} \leq \norm{ n } \sqrt{\epsilon}.
\end{align*}
As the left-hand side does not depend on $\epsilon$ and $\epsilon$ was arbitrary, we conclude $\Phi(n^*m)(x) = 0$ as desired.

Now assume that $x \in \mathrm{dom}(n) \cap \mathrm{dom}(m)$. Multiplying the equation \[ n n^* (b\circ \alpha_{n^*})=n b n^*\]
for $b\in B$ 
by $n^*$ on the left yields the following equation in $A$:

\begin{equation}\label{eq:commutation}
	n^* [n n^* \cdot (b\circ \alpha_{n^*})]
	=
	n^* n b n^*
	=
	b (n^* n ) n^* 
	,
\end{equation}
where the last equation follows from the fact that $n^* n \in B$, so that it commutes with $b$. 
Similarly, 
\begin{equation*}
    b n^* n n^*  m m^* m
    =
    n^* [n n^*  \cdot  (b\circ \alpha_{n}\inv)] m m^* m 
    =
    n^* m \bigl[ m^* m  \cdot [n n^* \cdot (b\circ \alpha_{n}\inv)]\circ \alpha_m  \bigr]
    ,
\end{equation*}
where we obtain the final equality by applying Equation \eqref{eq:commutation} with $n$ replaced by $m^*$ and $b$ replaced by $[nn^* \cdot (b \circ \alpha_{n^*})].$
In particular, $B$-linearity of $\Phi$ yields
\begin{align}   
    b \cdot n^* n \cdot \Phi(n^*  m) \cdot m^* m
    &=
    \Phi(b n^* n n^*  m  m^* m)
    \notag
    \\
    &=
    \Phi(n^* m \bigl[ m^* m  \cdot [n n^* \cdot (b\circ \alpha_{n}\inv)]\circ \alpha_m  \bigr])
    \notag
    \\
    &=
    \Phi (n^* m) \cdot \bigl[ m^* m  \cdot [n n^* \cdot (b\circ \alpha_{n}\inv)]\circ \alpha_m  \bigr]\label{eq:fct-last-line},
\end{align}
and this is an equation of globally defined continuous functions on $\widehat{B}$.
Since $n^*n (x) > 0$ and $m^*m (x) > 0$ by assumption, then if $b\in B$ is any element with $b(x)\neq 0$, then $\Phi(n^*  m) (x) \neq 0$ if and only if the function in Equation~\eqref{eq:fct-last-line} is non-zero when evaluated at $x$,
which implies
\[
   [n n^* \cdot (b\circ \alpha_{n}\inv)] \circ \alpha_m (x)
   \neq 0
   .
\]
This can happen only if $\alpha_m (x)\in \mathrm{dom} (n^*)$ and $b \bigl(\alpha_{n}\inv (\alpha_m (x))\bigr) \neq 0.$

In particular, if $\alpha_m (x) \notin \mathrm{dom}(n^*)$, then it follows that $\Phi(n^*m)(x)=0$ as claimed, by taking any $b\in B$ with $b(x)\neq 0$.
On the other hand, if $\alpha_m (x) \in \mathrm{dom}(n^*)$, so that $\alpha_n\inv (\alpha_m (x))$ makes sense and is by assumption not equal to $x$, then we can find a function $b\in B$ such that $b(x) \neq 0$ and $b(\alpha_n\inv (\alpha_m (x)))=0$. Again, the previous argument shows that we must have $\Phi (n^*m)(x)=0$, as claimed.
\end{proof}

The next result is a straightforward consequence of Lemma \ref{lem:nbhd for alpha_i}.

\begin{cor}\label{lem:first}
    Let $n, m \in N(B)$. If either
    \begin{enumerate}[label=(\arabic*), font=\upshape]
        \item\label{item:not-in-domain=} $x \notin \mathrm{dom}(n) \cap \mathrm{dom}(m)$, or
        \item\label{item:in-domain=} $x \in \mathrm{dom}(n) \cap \mathrm{dom}(m)$ satisfies $\alpha_n(x) \not= \alpha_m(x)$,
    \end{enumerate}
     then $\| n-m\|^2 \geq n^*n(x)$.
\end{cor}

The following lemma is readily checked using the defining equation for $\alpha_m$,  Equation~\eqref{eq:defining-eq-for-alpha}. However, we warn the reader that the analogous statement in the Weyl twist only holds if    $\lambda > 0$.  See Lemma \ref{lem:general-Sigma-description} below.
\begin{lemma}\label{lem:Annas-former-footnote-lemma}
    For any $\lambda \in \mathbb{C}\setminus\{0\}$ and $m\in N(B)$, we have $\alpha_m = \alpha_{\lambda m}$. In particular, for any $z \in \mathrm{dom} (m),$
\begin{equation}\label{eq:scalars-dont-matter-wrt-sim}
    [\alpha_m(z), m, z] = [\alpha_{\lambda m}(z), \lambda m, z]
    .
\end{equation}
\end{lemma}

\begin{lemma} 
\label{lem:alpha_i}
    Suppose $n \in N(B)$ and $n=\displaystyle \sum_{i \in F}n_i$ with $F$ finite and each $n_i \in N(B)$. For every $x\in \mathrm{dom}(n)$, there is an $i\in F$ and an open set $U\subseteq \mathrm{dom}(n)\cap \mathrm{dom}(n_i)$ containing $x$ such that $\alpha_{n_i}\vert_U \equiv \alpha_n \vert_U$.
\end{lemma}

\begin{proof}
Suppose, seeking contradiction, that there exists $x \in \mathrm{dom}(n)$ such that for every $i$ in $I:= \{i \in F \text{ such that } x\in \mathrm{dom} (n_i)\}$ and every neighborhood $U$ of $x$ in $\mathrm{dom}(n)\cap \mathrm{dom}(n_i)$, there exists an element $x_{U,i}$ such that $\alpha_{n_i} (x_{U,i}) \neq \alpha_{n} (x_{U,i})$. 

By Lemma~\ref{lem:nbhd for alpha_i}, we have $\Phi(n^* n_i) (x_{U,i}) = 0$. Since, for any fixed $i\in I$, the net $(x_{U,i})_U$ converges to $x$, we conclude from the continuity of $\Phi(n^* n_i) \in C_0(\widehat B)$ that $\Phi(n^* n_i) (x)=0$ also.

On the other hand, we have $x\notin \mathrm{dom} (n_j)$ for every $j\in F \setminus I$ by definition, so Lemma~\ref{lem:nbhd for alpha_i}\ref{item:not-in-domain=Phi-zero} implies $\Phi (n^* n_j) (x)=0$ also.

All in all, we have proved $\Phi (n^* n_i) (x)=0$ for all $i\in F$, so that
\[
    n^*n (x)
    =
    \Phi(n^*n) (x)
    =
    \sum_{i\in F} \Phi (n^* n_i) (x)
    =
    0,
\]
which contradicts the assumption that $x \in \mathrm{dom}(n)$.
\end{proof}

\begin{proof}[Proof of Proposition~\ref{prop:weyl-picture}]\label{pf:prop:weyl-picture}
    Let $n \in N(B)$ and suppose first that $n = \sum_{i\in F} n_i$ for a finite $F$ and $n_i = \lambda_i m_i$ with $m_i \in N$ and $\lambda_i \in \mathbb{C}\setminus \{0\}$. Then for any $ x \in \mathrm{dom}(n)$, we know by Lemma \ref{lem:alpha_i} that there exists an $i\in F$ and neighborhood of $ x $ on which $\alpha_{n}$ and $\alpha_{n_i}$ agree.

Since $\alpha_{\lambda_i m_i} = \alpha_{m_i}$, we conclude that \[
  [\alpha_n( x ), n, x ]
  =
  [\alpha_{m_i}( x ), m_i, x ]
  .
\]

Next, take an arbitrary element of the Weyl groupoid, say $[\alpha_n( x ), n, x ]$ for $n \in N(B)$ where $x \in \widehat B$ is such that $n^*n(x) >0$. Write $n = \lim_{q\to \infty} n_q$ where each $n_q$ is a finite linear combination of elements from $N$. 
Fix a compact neighborhood $K$ of $x$ such that $n^*n(y) > 0$ for $y\in K$, and write $\epsilon = \inf \{ | n^*n(y)|^{1/2}: y \in K\}$.  Observe that $\epsilon > 0$ since $K$ is compact. Let $Q$ be large enough so that $\| n - n_q\| < \epsilon$ for all $q \geq Q$. In particular, for any $y\in K^o$, we have $\norm{ n - n_q}^2 < n^*n (y)$. Therefore, for these $y$, we must have that $y\in \mathrm{dom}(n_q)\cap \mathrm{dom}(n)$ and $\alpha_n (y) = \alpha_{n_q} (y),$ by Corollary~\ref{lem:first}.

In other words, $\alpha_n$ and $\alpha_{n_q}$ agree on $K^o$, so $[\alpha_n(x), n, x] = [\alpha_{n_q}(x), n_q, x ]$.  Since $[\alpha_{n_q}(x), n_q, x]=[\alpha_m(x), m, x]$ for some $m \in N$ by the first part of the proof, we are done.
\end{proof}
\section{Computing the Weyl groupoid and twist in the group setting} \label{sec:anna}
Let $G$ be a countable discrete group, 
and let $c$ be a 2-cocycle on $G$. Suppose $S $ is maximal among abelian subgroups of $G=\Iso (G) $ on which $c$ is symmetric, and assume further that $S$ is 
normal and immediately centralizing. By Theorem~\ref{thm:main}, the pair $(A,B):=( C^{*}_{r}(G, c) , C^{*}_{r}(S, c))$ is Cartan.

In this section, we describe the relationship between the Weyl groupoid $\G_{(A,B)}$ and Weyl twist $\Sigma_{(A, B)}$ associated to $(A, B)$ via Renault's construction, and our original data $(G, S, c)$.  See Theorem \ref{lem:class-of-g-determines-equality} and Theorem \ref{thm:general-untwist-case} below.

For any $g \in G$, the Dirac-delta function $\delta_{g}$ on $G$ is in $N(B)$ since $S$ is normal: indeed, we have
\begin{equation}\label{eq:delg-dels-delg-star}
	\delta_{g} \delta_{s} \delta_{g}^{*}
	=
	\overline{c(g,g \inv)}\ c(s, g\inv) \ c(g, sg\inv) \ 
	\delta_{gsg\inv}
	\in C_{c}(S),
\end{equation}
and so we conclude that $\delta_{g} f \delta_{g}^{*} \in C_{r}^{*}(\S,c)$ for all $f\in C_{r}^{*}(\S,c)$.

As such elements densely span $A$, Proposition \ref{prop:weyl-picture} shows that they entirely determine the family of partial homeomorphisms $\{\alpha_{n} : n\in N(B)\}$ in the sense that every element of the Weyl groupoid $\G:=\G_{(A,B)}$ is of the form $[\alpha_{\delta_{g}}(  x   ), \delta_{g},   x   ]$ for some $g\in G$ and $  x   \in \widehat{B}$; we will therefore abuse notation from now on and write $\alpha_{g}$ instead of $\alpha_{\delta_{g}}$ whenever this is unambiguous. Since $\delta_{g}^{*} \delta_{g} = \delta_e$, we have $\textrm{dom}(\delta_{g}) = \widehat{B},$ i.e.\ $\alpha_{g}$ is globally defined, and because of Equation~\eqref{eq:scalars-dont-matter-wrt-sim}, the groupoid composition can be rephrased as:
\[
    [\alpha_{h}(\alpha_{g}(  x   )), \delta_{h}, \alpha_{g}(  x   )]
    \cdot
    [\alpha_{g}(  x   ), \delta_{g},   x   ]
    =
    [\alpha_{hg}(  x   ), \delta_{hg},   x   ].
\]
In particular,
\begin{equation}
    \label{eq:source-and-range-of-Weyl}
    s([\alpha_{g}(  x   ), \delta_{g},   x   ]) =   x   
    \text{ and }
    r([\alpha_{g}(  x   ), \delta_{g},   x   ]) = \alpha_{g}(  x   ).
\end{equation}

In what follows, we will use the usual notation $G/S$ for the quotient of $G$ by the normal subgroup $S$.  However, to simplify certain computations, we will usually think of $[g] \in G/S$ as denoting the {\em left} coset of $S$ with respect to $g$, which equals the right coset because $S$ is normal.

\begin{lemma}\label{lem:alph-only-depends-on-class}
	If $[g]=[h]$ in $\under{S}{G}$, then  $\alpha_{g}\equiv \alpha_{h}$.
\end{lemma}

Consequently, we will sometimes write $\alpha_{[g]}$ for $\alpha_{g}$.

\begin{proof}
	Recall that $\alpha_{g}$ is uniquely determined as satisfying Equation~\eqref{eq:defining-eq-for-alpha} for $n=\delta_{g}$. To show that $\alpha_{g}=\alpha_{sg}$ for every $s\in S$, it thus suffices to check that
	$
		\delta_{sg}^{*} b \delta_{sg}
=
		\delta_{g}^{*} b \delta_{g}
	$ for all $b \in B$. 
	As $\delta_{sg} = \overline{c(s,g)}\delta_{s} \delta_{g},$
	the left-hand side can be rewritten as $\delta_{g}^{*} (\delta_{s}^{*} b\delta_{s} ) \delta_{g}$. The fact that $B$ is commutative and $s\in S$ implies $\delta_{s}^{*} b\delta_{s} = b \delta_{s}^{*}\delta_{s} = b$, so the left-hand side of the equation indeed equals the right-hand side. 
\end{proof}

Note that $[g]\mapsto \alpha_{[g]}$ defines an action of $\under{S}{G}$ on $\widehat{B}$ by homeomorphisms since 
$
    \alpha_{[g]}\circ\alpha_{[h]} 
    =\alpha_{[gh]}
$
by Lemma~\ref{lem:Annas-former-footnote-lemma}.
Since $G$ 
is discrete, the transformation groupoid $\mathcal{K}:=  (\under{S}{G}) \tensor[_\alpha]{\ltimes}{} \widehat{B}$ 
is {\'e}tale (see \cite[Ex.~2.4.5]{AidansNotes}). It is furthermore locally compact Hausdorff because   $\widehat{B}$ is locally compact Hausdorff.
Lastly, it is second countable because $S$ is countable. The following Theorem shows that under mild hypotheses, $\K$ is the Weyl groupoid $\G_{(A,B)}.$

Recall that $\alpha$ is called \emph{topologically free} if for every finite set $F \subseteq \left(\under{S}{G} \right)\setminus \{e\}$, the set $\{x \in \widehat{B} \, | \, \forall\ g \in F,\ \alpha_g(x) \neq x\}$ is dense in $\widehat{B}$.
\begin{thm}\label{lem:class-of-g-determines-equality}
    If $\alpha$ is topologically free, then the map
    \[
      \varphi:  \mathcal{K} \to \G_{(A,B)} , 
        \quad
       \varphi ([g],x)  
       = [\alpha_{g}(  x   ), \delta_{g},   x   ] ,
    \]
    is an isomorphism of {topological} groupoids.
\end{thm}

\begin{proof}
    To see that $\varphi$ is well-defined, note that Lemma \ref{lem:alph-only-depends-on-class} shows that $[g] = [h]$ implies $\alpha_{g} = \alpha_{h}$ on all of $\widehat{B}$. So, in particular, $[\alpha_{g}(  x   ), \delta_{g},   x   ] = [\alpha_{h}(  x   ), \delta_{h},   x   ] \in  \G_{(A,B)}$ for any $  x    \in \widehat{B}$ by definition of $\sim$ (see Equation~\eqref{def:sim}).
    
    If $[g] \neq [h]$, then for any neighborhood $U$ of $x\in\widehat{B}$, the set $\{y \in U \, | \, \alpha_{g\inv h}(y) \neq y\}$ is nonempty   since $g\inv h \neq e$ and since $\alpha$ is topologically free by assumption.     Therefore, $\alpha_{g} \neq \alpha_{h}$ on $U$, which implies $[\alpha_{g}(  x   ), \delta_{g},   x   ] \neq [\alpha_{h}(  x   ), \delta_{h},   x   ].$ In particular, the map $\varphi$ is injective. Note that $\varphi$ is surjective because every element of $\G_{(A,B)}$ has a representative of the form $(\alpha_g(x),\delta_g, x)$ by Proposition \ref{prop:weyl-picture}.

    To see that $\varphi$ is a groupoid homomorphism, we compute on the one hand,
    \[
        ([g],\alpha_{h}(y))\cdot ([h],y)
        =
        ([gh],y)
    \]
    and on the other hand
    \[
        [\alpha_{g}(  \alpha_{h}(y)  ), \delta_{g},   \alpha_{h}(y)   ] 
        \cdot
        [\alpha_{h}(  y   ), \delta_{h},   y   ] 
        =
        [\alpha_{gh}(y)  , \delta_{gh},   y   ] 
        .
    \]
    Thus, $\varphi$ is a groupoid isomorphism.

   To see that $\varphi$ 
   is a homeomorphism, recall (cf.~\cite[Section 3]{renault-cartan}) that a basic open set in $\G_{(A, B)}$ is of the form $\{ [\alpha_n(x), n, x]: \alpha_n(x) \in V, x \in U\}$ for $U, V \subseteq \widehat B$ open and $n \in N(B)$.
   Consequently, the fact that $\alpha_g$ is a globally defined homeomorphism implies that every point $[\alpha_g(x), \delta_g, x] \in \G_{(A,B)}$ has an open neighborhood $\mathcal O$ of the form $\mathcal O =\{ [\alpha_g(y), \delta_g, y] : y \in U\}$ for some open set $U \subseteq \widehat B$.  Observe that 
   \[  
        \varphi^{-1}(\mathcal O) =
        \{ ([g],y): y\in U\},
    \]
    which is open in $\K$ since $ \under{S}{G}$ has the discrete topology.  Thus, $\varphi$ is continuous.
   
   To see that $\varphi$ is open, observe that $\varphi$ takes any basic open set $\{[g]\} \times U$ 
   in $\K$ (where $U \subseteq \widehat B$ is open)  to the basic open set $\{ [\alpha_g(y), \delta_g, y ] : y \in U\}$ in $\G_{(A,B)}$. This completes the proof that $\varphi: \K \to \G_{(A,B)}$ is an isomorphism of topological groupoids.
\end{proof}

\begin{rmk}
    The assumption of topological freeness in Theorem~\ref{lem:class-of-g-determines-equality} was only needed to prove injectivity of the map $\varphi$.
\end{rmk}

Next we turn to the Weyl twist $\Sigma:=\Sigma_{(A,B)}$ associated to the Cartan pair $(A,B)=(C^{*}_{r}(G, c), C^{*}_{r}(S, c))$. We will show that the twist is given by a 2-cocycle. That is, it is isomorphic as a topological groupoid to $\mathcal{K}  \times_{\sigma} \mathbb{T}$ for $\sigma$ a 2-cocycle on $\mathcal{K}$.

In \cite[Lemma 4.16]{renault-cartan}, Renault gives a family of local trivializations for $\Sigma$, indexed by $\{ n \in N(B)\}$.  These are defined on the {\em open support} of $\alpha_{n}$,
\[ \text{supp}'(n) = \left\{ [\alpha_{n}(  y  ), n,   y  ] \in \G  :   y   \in \textup{dom}(n)\right\}.\]

The discussion preceding \cite[Lemma 4.16]{renault-cartan}
describes the trivialization of the $\mathbb{T}$-bundle $\Sigma$ over $\G$ as follows: for each normalizer $n \in N(B)$, the homeomorphism $\phi_{n}\colon  \Sigma|_{\supp'(n)} \to \textup{dom}(n) \times \mathbb{T}$ is given by $\phi_{n}\inv(  x   ,  \lambda ) =  \llbracket \alpha_{n}(  x   ),  \lambda  n,   x    \rrbracket$, where $ \lambda  \in \mathbb{T}$.\label{page:def:phi-g}

\begin{lemma}
    \label{lem:general-Sigma-description}
	Let $R(S)$ be any choice of coset representatives for $\under{S}{G}$. Then every element of $\Sigma$ can be represented by some
    $(\alpha_{g} (  x   ),  \lambda \delta_{g},   x   )$, where $g\in R(S)$ and where $ \lambda \in \mathbb{T}$ can be explicitly computed.
\end{lemma}

\begin{proof}
    Let $\llbracket\alpha_{n}(  x   ), n,   x   \rrbracket$ be an arbitrary element of $\Sigma.$ We know from Theorem \ref{lem:class-of-g-determines-equality} (cf.~also  Proposition \ref{prop:weyl-picture}) that there exists a (unique) $g\in R(S)$ such that
    \[
        [ \alpha_{n}(  x   ), n,   x   ]
        =
        [ \alpha_{g}(  x   ), \delta_{g},   x   ]
        ,
    \]
    i.e.~there exists a neighborhood $U$ of $  x   $ on which $\alpha_{n}$ and $\alpha_{g}$ coincide. In other words, $\alpha_{n^{*}}$ and $$\alpha_{g\inv} := \alpha_{\delta_{g\inv}} = \alpha_{{c(g,g\inv)} \delta_{g}^*} = \alpha_{ \delta_{g}^*} $$
    coincide on $\alpha_{n} (U)$. If $f\in B$ 
    vanishes outside of $U$, then $f\circ \alpha_{n^{*}} = f\circ \alpha_{g\inv}$ is a globally defined function, and we can use Lemma \ref{lem:stronger-def-for-alpha} together with the definition of $\alpha$ (Equation \eqref{eq:defining-eq-for-alpha}) and the fact that $\delta_{g\inv}^{*} \delta_{g\inv} = \mathrm{id}$ to compute that 
    \begin{align*}
        nf 
        &=
        (f\circ \alpha_{n^{*}}) n
        =
        (f\circ \alpha_{g\inv})n
        =
       \delta_{g} f \delta_{g}^{*} n
        .
    \end{align*}
    Let us check that $f' := f \delta_{g}^{*} n$ is actually an element of $B$: First, if $b\in B$ is arbitrary, then
    \begin{align}
        \begin{split}\label{eq:a-helpful-small-comp}
        f' b f'^{*}
        &=
        (f\delta_{g}^{*} n) b (n^{*} \delta_{g} f^{*}) 
        =
        f \delta_{g}^{*} \left( (b\circ \alpha_{n^{*}}) nn^{*} \right) \delta_{g}f^{*} 
        \\
        &=
        f f^{*} \cdot \left( (b\circ \alpha_{n^{*}}) nn^{*} \right)\circ \alpha_{\delta_{g}}
        =
        f f^{*} b \left( nn^{*} \circ \alpha_{\delta_{g}} \right),
    \end{split}
    \end{align}
    where the last equality follows from the fact that, on the support of $f$, we have $ \alpha_{n^{*}}\circ \alpha_{\delta_{g}} = \mathrm{id}$. On the other hand, since $n\in N(B)$, we have $f'\in N(B)$ and hence by definition of $\alpha_{f'}$:
    \begin{align*}
        f' b f'^{*}
        &=
        (b\circ \alpha_{f'^{*}}) \cdot f'f'^{*}
        =
        (b\circ \alpha_{f'^{*}}) \cdot (f \delta_{g}^{*} n) (n^{*} \delta_{g} f^{*})
        \\
        &=
        (b\circ \alpha_{f'^{*}}) f f^{*} \left( nn^{*} \circ \alpha_{\delta_{g}} \right).
    \end{align*}
    Comparing this with Equation~\eqref{eq:a-helpful-small-comp}, we conclude by uniqueness that $\alpha_{f'^{*}}$ must be the identity on its domain. It follows from Lemma \ref{lem:alpha-n=id-implies-n-in-B} that $f'$ is an element of $B$.
    
    As $\widehat{B}$ is locally compact and Hausdorff,  we can choose a compactly supported function $f$ which vanishes outside of $U$ and has $f(  x   )\neq 0$. We then have (similarly to Equation \eqref{eq:a-helpful-small-comp})
    \begin{align*}
        (f'f'^{*}) (  x   )
        =
        (f f^{*})(  x   ) \cdot nn^{*}\left( \alpha_{\delta_{g}} (  x   )\right)
        =
        (f f^{*})(  x   )\cdot nn^{*}\left( \alpha_{n} (  x   )\right)
        .
    \end{align*}
    This is non-zero since $  x   \in \mathrm{dom}(n)$, and we conclude that $f'(  x   )\neq 0$.

     Now that we have found $f, f' \in B$ such that $f(  x   ) \neq 0 \neq f'(  x   )$ and $nf = \delta_{g} f',$ we can define $b,b'\in B$ and $\lambda\in\mathbb{T}$ by
    \[
        b := \frac{\abs{f(  x   )}}{f(  x   )} f,
        \;
        b' := \frac{\abs{f'(  x   )}}{f'(  x   )} f',\;
        \text{ and }
         \lambda  := \frac{\abs{f(  x   )}}{f(  x   )}\frac{f'(  x   )}{\abs{f'(  x   )}},
    \]
    so that  $b(  x   ) >0$ and $b'(  x   )>0$ and:
    \begin{align*}
        nb
        &=
        n \left( \frac{\abs{f(  x   )}}{f(  x   )} f \right)
        =
        \frac{\abs{f(  x   )}}{f(  x   )} \cdot \delta_{g} f'
        =
        \frac{\abs{f(  x   )}}{f(  x   )} \cdot \delta_{g} \left( \frac{f'(  x   )}{\abs{f'(  x   )}}\right)b'
        =
         \lambda  (\delta_{g} b')
        .
    \end{align*}
    We conclude
    $
        \llbracket \alpha_{n}(  x   ), n,   x   \rrbracket
        =
        \llbracket \alpha_{g}(  x   ),  \lambda \delta_{g},   x   \rrbracket
        ,
    $
    as claimed.
\end{proof}

Our next goal is to show that under the hypotheses of Theorem \ref{lem:class-of-g-determines-equality}, the Weyl twist arises from a 2-cocycle $\sigma$ on $\K.$  First, we will show that $\Sigma$ is a trivial circle bundle; that is, that $\Sigma$ is homeomorphic to $\K \times \mathbb{T}.$

\begin{thm}
\label{thm:general-Sigma}
    Let $G$ be a countable discrete group with 2-cocycle $c$. Suppose $S $ is maximal among abelian subgroups of $G$ on which $c$ is symmetric, and assume further that $S$ is normal and immediately centralizing. Let $\Sigma$ denote the Weyl twist associated to the Cartan pair $( C^{*}_{r}(G, c) , C^{*}_{r}(S, c))$. If the action $\alpha$ on $\widehat B$ is topologically free, then the map $\psi\colon  \mathcal{K}  \times \mathbb{T} \to \Sigma$ given by
    \[
    \psi
   \bigl(   [g], x   ,   \lambda \bigr)
    := 
    \llbracket\alpha_{[g]}(  x   ),  \lambda  \delta_{r_{[g]}},   x   \rrbracket,
    \]
where $ \lambda \in \mathbb{T}$ and $r_{[g]} \in R(S)$ is our chosen representative of $[g]$, is a homeomorphism of topological spaces.
\end{thm}

\begin{proof}
 Surjectivity of $\psi$ follows from Lemma \ref{lem:general-Sigma-description}.
 
 For injectivity, take two elements $(  [g],  x    ,\lambda)$ and $( [h], y   ,  \nu)$ 
 of $\mathcal{K} \times \mathbb{T}$.
 If $  x    \neq   y   $, then the source of $\psi\bigl(  [g],   x   ,\lambda \bigr)$ 
 is $  x   $, while the source of $\psi\bigl(   [h], y   , \nu \bigr)$ 
 is $  y   $ (see Equation \eqref{eq:source-and-range-of-Weyl}), and thus $\psi(  [g],x    ,  \lambda) \neq \psi(  [h], y   ,  \nu).$  
 
 So assume $  x    =   y   $. If $\psi(   [g], x    ,\lambda) = \psi( [h],  x   , \nu),$ 
 then there must exist $b, b'\in B$ such that $b(  x   ), b'(  x   ) >0$ and $\lambda \delta_{r_{[g]}} b = \nu \delta_{r_{[h]}} b'$. Rearranging yields
 \[
 	B \ni \overline{\lambda} \nu b' = \delta_{r_{[h]}}^{*} \delta_{r_{[g]}} b =
   \overline{c(r_{[h]}, r_{[h]}\inv)} c(r_{[h]}\inv, r_{[g]}) \delta_{r_{[h]}\inv r_{[g]}} b
    .
 \]
 In particular, $\delta_{r_{[h]}\inv r_{[g]}} b \in B$. If we consider $C^{*}_{r}(G, c)$ as a subset of $C_{0}(G)$ (see \cite[Prop.~II.4.2]{renault}), then for $k\in G$, the formula for the convolution yields
 \[
    (\delta_{r_{[h]}\inv r_{[g]}} b) (k)
    =
    b(r_{[g]}\inv r_{[h]} k) c(r_{[h]}\inv r_{[g]}, r_{[g]}\inv r_{[h]} k)
    .
 \]
Recall from Lemma~\ref{lem:Bs-elts-have-supp-in-S} that an element of $A= C_{r}^{*} (\G,c)$ is in $B=C_{r}^{*}(\S,c)$ exactly if its image in $C_{0}(\G)$ has support in $\S$.
 As $B\ni b\neq 0$, there therefore exists an element $s\in S$ with $b(s)\neq 0.$ Plugging in $k:=  r_{[h]}\inv r_{[g]} s$ 
yields:
 \[
    (\delta_{r_{[h]}\inv r_{[g]}} b) (r_{[g]} r_{[h]}\inv s)
    =
    b(s) c(r_{[h]}\inv r_{[g]}, s) 
    \neq 0.
 \]
 But since $\delta_{r_{[h]}\inv r_{[g]}} b \in B$ is also supported in $S$, this implies that $r_{[g]} r_{[h]}\inv s \in S$. It follows that $[r_{[g]}]= [r_{[h]}]$ in $\under{S}{G}$, so that $r_{[g]}=r_{[h]}$ since $R(S)$ contains exactly one representative for each class in $\under{S}{G}$.
 Thus $\lambda \delta_{r_{[g]}} b = \nu \delta_{r_{[h]}} b'$ implies $\lambda b = \nu b'$. 
 Since $b(  x   )$ and $b'(  x   )$ were assumed to be positive and $\lambda ,\nu \in \mathbb{T}$, this forces $\lambda = \nu$. This concludes our proof of injectivity.

Continuity of $\psi$ follows from the fact that each $\phi_{g}\inv := \phi_{\delta_{g}}\inv$ is continuous (in fact a homeomorphism), and that $\mathcal{K} $ is {\'e}tale since $\under{S}{G}$ is discrete. To be precise, if $(  [g_{i}], x_{i}, t_{i}) \to (  [g],  x   , t)$, 
we must have $  x_{i} \to   x  $ and $t_{i} \to t,$ and for large enough $i$, $[g_{i}] = [g]$ by discreteness. In particular $r_{[g_{i}]}=r_{[g]}$, so for notational convenience, let us denote $g=r_{[g]}$. The continuity of $\phi_{g}\inv$ implies that $\llbracket \alpha_{g}(  x_{i}), t_{i} \delta_{g},   x_{i} \rrbracket \to \llbracket \alpha_{g}(  x   ), t \delta_{g},   x   \rrbracket,$ so $\psi$ is continuous.

For the continuity of $\psi\inv$, suppose that $\gamma_{i}\to \gamma$ in $\Sigma$. By Lemma~\ref{lem:general-Sigma-description}, we know that $\gamma$ can be represented by $(\alpha_{g} (  x   ), \lambda\delta_{g},   x   )$ for some $x\in\widehat{B}$, $g\in R(S)$, and $\lambda\in\mathbb{T},$ and similarly $\gamma_{i}$ can be represented by  $(\alpha_{g_{i}} (  x_{i}   ), \lambda_{i}\delta_{g_{i}},   x_{i}   )$. 
In particular, since we have chosen $ g_{i}\in R(S),$
\begin{align}
    \label{eq:gamma-i}
    \gamma_{i}
    &=
    \llbracket \alpha_{g_{i}}(  x_{i}   ), \lambda_{i} \delta_{g_{i}},   x_{i}    \rrbracket
    =
    \llbracket \alpha_{g_{i}}(  x_{i}   ), \lambda_{i} \delta_{r_{[g_{i}]}},   x_{i}
    \rrbracket
    =
    \psi ([g_{i}], x_{i}, \lambda_{i}) 
    .
\end{align}

By convergence, we have for large enough $i$ that $\gamma_{i}$ is in the basic open neighborhood
\[\phi_{g}\inv(\mathrm{dom}(\delta_{g}) \times \mathbb{T}) = \phi_{g}^{-1}(\widehat{B} \times \mathbb{T}) = \Sigma_{\supp'(\lambda\delta_{g})}\]
of $\gamma=\llbracket \alpha_{g}(  x   ), \lambda \delta_{g},   x    \rrbracket. $  By definition of $\phi^{-1}$
there must hence exist some $\nu_{i} \in \mathbb{T}$ so that
\begin{alignat*}{2}
    \gamma_{i}
    &=
    \llbracket \alpha_{g}(  x_{i}   ), \nu_{i} \delta_{g},   x_{i}    \rrbracket
    \\
    &=
    \llbracket \alpha_{g}(  x_{i}   ), \nu_{i} \delta_{r_{[g]}},   x_{i}    \rrbracket
    &\quad \text{(since $g\in R(S)$)}
    \\
    &=
    \psi ([g], x_{i}, \nu_{i}). 
\end{alignat*}
As we have shown $\psi$ to be injective, we conclude from Equation~\eqref{eq:gamma-i} that $\nu_{i}=\lambda_{i}$ and that $[g]=[g_{i}]$, so $g=g_{i}$ as these elements were chosen to be in $R(S)$.

Because the topology of $\Sigma$ is inherited from that of $\text{dom}(\delta_{g}) \times \mathbb{T}$ via the maps $\phi_{g},$ our hypothesis that $\gamma_{i}
\to \gamma$ implies that  
\begin{align*}
&\phi_{g}\inv
\left(
    \llbracket \alpha_{g}(  x_{i}   ), \nu_{i} \delta_{g},   x_{i}    \rrbracket
\right)
= \phi_{g}\inv
\left(
    \gamma_{i}
\right)
\longrightarrow
\phi_{g}\inv
\left(  \gamma \right)
=
\phi_{g}\inv
\left(
    \llbracket \alpha_{g}(  x   ), \lambda \delta_{g},   x    \rrbracket
\right).
\end{align*}
By the definition of $\phi_{g}$ (see page~\pageref{page:def:phi-g}), this means exactly that the sequence $(x_{i},\lambda_{i})=(x_{i},\nu_{i})$ converges to $(x,\lambda)$ in $\widehat{B}\times\mathbb{T}$. Thus, for $i$ large enough so that $g_{i} = g,$ we have by Equation~\eqref{eq:gamma-i},
\[
    \psi\inv (\gamma_{i})
    =
    ([g_{i}], x_{i},  \lambda_{i}) 
    =
    ([g], x_{i}, \lambda_{i}) 
    \to
    ([g],x,  \lambda) 
    =
    \psi\inv (\gamma).
    \qedhere
\]
\end{proof}

Knowing that the bundle structure on $\Sigma$ is  trivial, we now compute the 2-cocycle on $\K = (\under{S}{G})\tensor[_\alpha]{\ltimes}{} \widehat B $ 
which gives rise to $\Sigma$. 
\begin{lemma}\label{lem:def:cocycle-on-mcK}
    For $y\in \widehat{B}$ and $[g],[h]\in \under{S}{G}$, define
\begin{align*}
    \sigma
    \bigl(
       (  [g],  \alpha_{h}(y)) , ([h],y) 
    \bigr)
    & :=
  \frac{ \left( \delta_{r_{[gh]}}^* \delta_{r_[g]} \delta_{r_{[h]}} \right)(y)}{\left| \delta_{r_{[gh]}}^* \delta_{r_[g]} \delta_{r_{[h]}} (y)\right|}.
\end{align*}
    Then $\sigma$ is a 2-cocycle on the groupoid $\K = (\under{S}{G})\tensor[_\alpha]{\ltimes}{} \widehat B $. 
\end{lemma}
\begin{proof}
Observe first that  $\sigma$ takes values in $\mathbb{T}$ by construction.  Furthermore, $\delta_{r_{[gh]}}^* \delta_{r_{[g]}} \delta_{r_{[h]}} \in B = C^*_r(S, c)$, because $[r_{[gh]}] = [r_{[g]} r_{[h]}] \in \under{S}{G}$.  Therefore, to check that $\sigma$ is a 2-cocycle, it suffices to check the cocycle condition.  This will follow if we can show  that for any $g, h, k \in G$, we have 
\begin{align}
\label{eq:anna-cocycle}
    \left(\delta_{r_{[ghk]}}^* \delta_{r_{[g]}} \delta_{r_{[hk]}} \right) (\alpha_k^{-1}(y)) & \left( \delta_{r_{[hk]}}^* \delta_{r_{[h]}} \delta_{r_{[k]}} \right) (\alpha_k^{-1}(y)) 
    \\
    &\overset{!}{=}\left(\delta_{r_{[ghk]}}^* \delta_{r_{[gh]}} \delta_{r_{[k]}} \right) (\alpha_k^{-1}(y)) \left( \delta_{r_{[gh]}}^* \delta_{r_{[g]}} \delta_{r_{[h]}} \right) (y) .
    \notag
    \end{align}
Since $y = \alpha_k(\alpha_k^{-1}(y))$ and $\delta_{r_{[k]}}^* \delta_{r_{[k]}} = \delta_e$ satisfies $\delta_{r_{[k]}}^* \delta_{r_{[k]}}(x) = 1$ for all $x \in \widehat{B}$, Equation \eqref{eq:defining-eq-for-alpha} implies for $\alpha_{r_{[k]}}=\alpha_{k}$ that 
\[ \delta_{r_{[k]}}^* \left(  \delta_{r_{[gh]}}^* \delta_{r_{[g]}} \delta_{r_{[h]}}\right) \delta_{r_{[k]}} (\alpha_k^{-1}(y)) = \left(  \delta_{r_{[gh]}}^* \delta_{r_{[g]}} \delta_{r_{[h]}}\right) (y).   \]
The fact that  Gelfand duality is a $*$-algebra homomorphism therefore implies that Equation \eqref{eq:anna-cocycle} will follow if we can show 
\[ 
\left( \delta^*_{r_{[ghk]} }\delta_{r_{[g]}} \delta_{r_{[hk]}}\right)
\,
\left( \delta_{r_{[hk]}}^* \delta_{r_{[h]}} \delta_{r_{[k]}}\right) 
\overset{!}{=}
\left(
    \delta_{r_{[ghk]}}^* \delta_{r_{[gh]}} \delta_{r_{[k]}}
\right)
\,
\left(\delta_{r_{[k]}}^* \delta_{r_{[gh]}}^* \delta_{r_{[g]}} \delta_{r_{[h]}} \delta_{r_{[k]}}\right),\]
and this follows immediately from the fact that $\delta_m$ is a unitary for all $m \in G$.
\end{proof}

\begin{thm}
\label{thm:general-Sigma-twist}
    In the setting of Theorem~\ref{thm:general-Sigma},
    the map $\psi$ is a groupoid homomorphism from $\mathcal{K} \times_{\sigma} \mathbb{T}$ to $\Sigma$ which intertwines the actions of $\mathbb{T}$.
\end{thm}

\begin{proof}
Note  that 
$
    \psi ([e],x,1) 
    = \llbracket \alpha_{[e]}(  x   ), \delta_{r_{[e]}},   x    \rrbracket
$
is an element of $\Sigma\z$ since $ \alpha_{[e]}(  x   )=x$, so $\psi$ preserves the unit space.

We compute on the one hand for $\xi  :=(  [g], \alpha_{h}(y)   )$ and $\zeta := ([h],y),$ 
\begin{align*}
    &\psi\left(
    ( \xi  , \lambda )(\zeta ,\nu )
    \right)
    =
    \psi
        ([gh],y , \sigma (\xi , \zeta)\,\lambda\nu ) 
    \\
    &\quad=\llbracket \alpha_{gh}(y ), \lambda\nu\,\sigma (\xi , \zeta)\,\delta_{r_{[gh]}}, y  \rrbracket
     \\
    &\quad=\llbracket \alpha_{g}(  \alpha_{h}(y)   ), \lambda\nu\,\sigma (\xi , \zeta) \,\delta_{r_{[gh]}}, y  \rrbracket
    .
\end{align*}
On the other hand,
\begin{align*}
    &\psi(  \xi , \lambda )\psi(\zeta,\nu )
    = \llbracket \alpha_{[g]}(  \alpha_{h}(y)   ), \lambda  \delta_{r_{[g]}},   \alpha_{h}(y)    \rrbracket \cdot  \llbracket   \alpha_{h}(y)   , \nu  \delta_{r_{[h]}}, y \rrbracket
     \\
    &\quad= \llbracket\alpha_{g}(  \alpha_{h}(y)   ), \lambda  \nu \,\delta_{r_{[g]}} \delta_{r_{[h]}}, y \rrbracket
    = \llbracket \alpha_{g}(  \alpha_{h}(y)   ), \lambda\nu \,c(r_{[g]}, r_{[h]})\,\delta_{r_{[g]}r_{[h]}}, y  \rrbracket
     .
\end{align*}
Let $n:=  \lambda\nu \,c(r_{[g]}, r_{[h]})\,\delta_{r_{[g]}r_{[h]}} = \lambda \nu \, \delta_{r_{[g]}} \delta_{r_{[h]}}$, and note that 
\[
    \alpha_{n} = \alpha_{r_{[g]}r_{[h]}} = \alpha_{gh},
\]
where the first equation is due to Lemma~\ref{lem:Annas-former-footnote-lemma} and the second due to Lemma~\ref{lem:alph-only-depends-on-class}, using that $S$ is normal. Thus, we have to show that
\begin{align}\label{eq:current-claim}
    \llbracket \alpha_{n}(  y   ), n, y \rrbracket
    \overset{!}{=}
    \llbracket \alpha_{r_{[gh]}}(  y   ), \lambda\nu\,\sigma (\xi , \zeta) \,\delta_{r_{[gh]}}, y \rrbracket
    .
\end{align}

Taking $f$ to be the constant function 1 in the proof of Lemma~\ref{lem:general-Sigma-description}, and observing that $\delta_{r_{[gh]}}^* n \in C_0(\widehat B)$ takes values in $\mathbb{T}$, reveals that  $
    \mu
    :=
    (\delta_{r_{[gh]}}^{*} n)(y)$
has the property
\begin{align*}
    \llbracket \alpha_{n}(  y   ), n, y \rrbracket
    =
    \big\llbracket \alpha_{r_{[gh]}}(  y   ), \tfrac{\mu}{\abs{\mu}} \, \delta_{r_{[gh]}}, y 
    \big\rrbracket
    =
    \llbracket \alpha_{r_{[gh]}}(  y   ), \mu\, \delta_{r_{[gh]}}, y \rrbracket
    .
\end{align*}
Indeed, the element $\delta_{r_{[gh]}}^{*} n$ of $C_{r}^{*}(G,c)$ is given by
\begin{align*}
    \delta_{r_{[gh]}}^{*} n
    =
    \lambda\nu \,
    \overline{c(r_{[gh]}\inv, r_{[gh]})} \,c(r_{[gh]}\inv, r_{[g]}r_{[h]})\,c(r_{[g]}, r_{[h]})\,\delta_{r_{[gh]}\inv r_{[g]}r_{[h]}}
    .
\end{align*}
Thus
\[
    \mu
    =
    (\delta_{r_{[gh]}}^{*} n)(y)
    =
    \lambda\nu\,
    \overline{c(r_{[gh]}\inv, r_{[gh]})} \,c(r_{[gh]}\inv, r_{[g]}r_{[h]}) \,c(r_{[g]}, r_{[h]})\,\delta_{r_{[gh]}\inv r_{[g]}r_{[h]}} (y).
\]
We now observe that 
\[  \overline{c(r_{[gh]}\inv, r_{[gh]})} \,c(r_{[gh]}\inv, r_{[g]}r_{[h]}) \,c(r_{[g]}, r_{[h]})\,\delta_{r_{[gh]}\inv r_{[g]}r_{[h]}} = \delta_{r_{[gh]}}^* \delta_{r_{[g]}} \delta_{r_{[h]}},
\]
revealing that $\mu $ is a positive scalar multiple of $ \lambda\nu\,\sigma (\xi , \zeta)$, so that 
Equation~\eqref{eq:current-claim} indeed holds.

Lastly, to see that $\psi$ intertwines the actions of $\mathbb{T}$ on $\mathcal{K}  \times_{\sigma} \mathbb{T}$ and $\Sigma$, we compute for $\nu  \in \mathbb{T}$:
\begin{align*}
    \psi(\nu  \cdot
        (  [g],  x   , \lambda ) 
    )
    &=
    \psi
     ( [g], x   ,  \nu\lambda) 
    = \llbracket \alpha_{g}(  x   ), \nu\lambda \delta_{r_{[g]}},   x    \rrbracket \\
&= \nu  \cdot \llbracket\alpha_{g}(  x   ), \lambda  \delta_{r_{[g]}},   x    \rrbracket = \nu  \cdot \psi
   (  [g],x   , \lambda ). 
\qedhere
\end{align*}
\end{proof}

\begin{thm}
\label{thm:general-untwist-case}
 With the assumptions from Theorem~\ref{thm:general-Sigma}, the Cartan pair
$(C_{r}^{*}(\mathcal{K} \times_{\sigma} \mathbb{T} ),C_0(\widehat{B}))$ is isomorphic to the pair $(C_{r}^{*}(G,c), C^{*}_{r}(S,c))$.
\end{thm}

\begin{proof}
	Theorems~\ref{lem:class-of-g-determines-equality} and  \ref{thm:general-Sigma} show that the twisted groupoid  $(\mathcal{K} , \mathcal{K}  \times_{\sigma} \mathbb{T})$ is isomorphic to the twisted groupoid $(\G_{(A,B)} , \Sigma_{(A,B)}),$ the Weyl twist associated to the Cartan pair $(A,B)=(C^{*}_{r}(G, c), C^{*}_{r}(S, c))$.  
	Thus, $C_{r}^{*}(\G_{(A,B)} , \Sigma_{(A,B)})$ is isomorphic to $C_{r}^{*}(\mathcal{K}, \mathcal{K}\times_{\sigma}\mathbb{T})$ in a diagonal-preserving fashion. By \cite[Theorem 5.9]{renault-cartan}, we thus have
	\begin{align*}
    	(C^{*}_{r}(G, c),C^{*}_{r}(S,c))
    	&\cong (C^{*}_{r}(\G_{(A,B)} , \Sigma_{(A,B)}),C_0(\G_{(A,B)}\z ))
    	\\
    	&\cong  (C^{*}_{r}(\mathcal{K}, \mathcal{K}\times_{\sigma}\mathbb{T}  ),C_0(\mathcal{K}\z ))\\
    	& \cong (C^*_r(\K, \sigma), C_0(\K\z)).
    	\qedhere
	\end{align*}
\end{proof}

\begin{rmk}
The isomorphism in Theorem \ref{thm:general-untwist-case} from $C_r^*(G,c)$ to $C_r^*(\mathcal{K}, \mathcal{K}\times_\sigma \mathbb{T})$ may be defined explicitly on the generators of $C_r^*(G,c)$ by
\begin{equation*}
    \delta_g \mapsto \widehat{\delta_g} \circ \psi, \quad g \in G,
\end{equation*}
where $\widehat{\delta_g} \in C_0(\Sigma) \subseteq C^*_r(\G, \Sigma)$ is defined by Renault in \cite[Lemma 5.3]{renault-cartan} by
\[ \widehat{\delta_{g}}(\llbracket \alpha_{n}(  y  ), n,   y   \rrbracket) := \frac{\Phi(n^{*}\delta_{g})(  y  )}{\sqrt{n^{*}n(  y  )}},\]
	and $\Phi \colon C^{*}_{r}(G, c) \to C^{*}_{r}(S , c)$ is the conditional expectation.
\end{rmk}

 {When the 2-cocycle $c$ is trivial on $S$, as will be the case in the examples considered in Section \ref{sec:apps} below, we can identify the Gelfand dual $\widehat B$ of the Cartan subalgebra  $B = C^*_r(S, c) \cong C^*_r(S)$ with the Pontryagin dual $\widehat S$ of $S$.}

	To be precise, if $c$ is {trivial} on $S$, 
	the map
	\[
		\Psi\colon \widehat{S} \overset{\cong}{\longrightarrow} \widehat{B},
		\quad
		\text{ determined by }
		\Psi (\nu) := \left[ C_{c} (S, c) \ni \ b \mapsto \sum_{s\in S} b(s) \nu (s) \right],
	\]
	is a homeomorphism with inverse
	\[
		\Psi\inv (\chi) = \bigl[ s \mapsto \mathrm{ev}_{\delta_{s}}(\chi):=\chi (\delta_{s}) \bigr].
	\]

\begin{prop}\label{lem:c_trivial_H} Suppose $c$ is trivial on $S$.  Then the action 	    $[g] \mapsto \tilde{\alpha}_{[g]}:=\Psi\inv \circ \alpha_{g} \circ \Psi$   of $\under{S}{G}$ on $\widehat{S}$ induced by $\Psi$   is given by

	  \[	\tilde{\alpha}_{[g]} (\nu) 
		=
		\Bigl[
	    	s\mapsto 
		    \overline{c(g, g\inv)} \, c(g\inv, s) \, c(g\inv s, g) \, \nu (g\inv s g)
		\Bigr]
		,
	\]
	for $[g]\in \under{S}{G}$, $\nu\in \widehat{S}$, and $s\in S$. Moreover the transformation groupoid
    $\H:=(\under{S}{G}) \tensor[_{\tilde{\alpha}}]{\ltimes}{} \widehat{S} $ 
	is isomorphic to 
	the	Weyl groupoid via
\[
        \G_{(A,B)} \to \mathcal{H} , 
        \quad
        [\alpha_{g}(  x   ), \delta_{g},   x  ] 
        \mapsto 
        ( [g], \Psi\inv(x) ). 
\]
\end{prop}

\begin{proof}
	Let us be very precise and carefully distinguish between $B$ and $C_{0}(\widehat{B})$; so we rewrite the defining equation for $\alpha_{g}$ 	as follows:
	\[
		\text{``For all } b\in B:
		\quad
		\mathrm{ev}_{\delta_{g}^{*} b \delta_{g}}
	    =
	    \mathrm{ev}_{b}\circ\alpha_{g}\text{.''}
	\]
	One easily verifies that $\tilde{\alpha}_{[g]}$ is indeed well-defined; that is, it only depends on $[g]$ and not on $g$.
	Now, using the defining property of $\alpha_g$ in the third equality, we compute for $\nu\in \widehat{S}$ and $s\in S$ that
	\begin{alignat*}{3}
	( \tilde{\alpha}_{[g]} (\nu))(s) 
		&=
		\left(\Psi\inv \circ \alpha_{g} \circ \Psi\right) (\nu) (s)
		=
		\mathrm{ev}_{\delta_{s}} \Bigl(\alpha_{g} \bigl(\Psi (\nu)  \bigr)\Bigr)	
		\\
		&=
		\mathrm{ev}_{\delta_{g}^{*} \delta_{s} \delta_{g}} \bigl(\Psi (\nu)  \bigr)
	=
		\Psi (\nu) \left(\delta_{g}^{*} \delta_{s} \delta_{g} \right)
		=
		\sum_{t\in S} \left(\delta_{g}^{*} \delta_{s} \delta_{g} \right)(t) \cdot \nu (t)
.
	\end{alignat*}
By Equation~\eqref{eq:delg-dels-delg-star}, we conclude the formula for $\tilde{\alpha}_{[g]}$. The claim about the isomorphisms follows directly from Theorem~\ref{lem:class-of-g-determines-equality} and the way we defined~$\tilde{\alpha}$.
\end{proof}

Thanks to Lemma~\ref{lem:def:cocycle-on-mcK}, the map
\begin{align}\label{eq:def:tilde-sigma}
    \tilde{\sigma} \bigl(
    ( [g],  \tilde{\alpha}_{h}(\nu) ) , ([h],\nu)
    \bigr)
    & :=
  \frac{ \left( \delta_{r_{[gh]}}^* \delta_{r_{[g]}} \delta_{r_{[h]}} \right)(\Psi(\nu))}{\left| \delta_{r_{[gh]}}^* \delta_{r_{[g]}} \delta_{r_{[h]}} (\Psi(\nu))\right|}.
\end{align}
for $\nu\in \widehat{S}$, $[g],[h]\in \under{S}{G}$, defines a 2-cocycle on the groupoid $\H=(\under{S}{G}) \tensor[_{\tilde{\alpha}}]{\ltimes}{} \widehat{S} $ 
if $c$ is trivial. Our previous results can be rephrased as follows:

\begin{cor}[of Theorems~\ref{thm:general-Sigma} and \ref{thm:general-Sigma-twist}]
\label{cor:general-Sigma-c-trivial-on-S}
In the setting of Theorem \ref{thm:general-Sigma}, assume moreover that $c$ is trivial on $S$. Then the map $\tilde{\psi}\colon  \mathcal{K}  \times_{\tilde{\sigma}} \mathbb{T} \to \Sigma$ given by
\[
    \tilde{\psi}
    \bigl(  [g], \nu   ,  \lambda\bigr) 
    := 
    \llbracket\alpha_{[g]}(  \Psi(\nu)   ), \lambda \delta_{r_{[g]}},   \Psi(\nu)   \rrbracket
    =
    \llbracket \Psi \bigl( \tilde{\alpha}_{[g]}(\nu)   \bigr), \lambda \delta_{r_{[g]}},  \Psi(\nu)   \rrbracket
    ,
\]
where $\lambda\in \mathbb{T}$ and $r_{[g]} \in R(S)$ is our preferred representative of $[g]$, is an isomorphism of topological groupoids.
\end{cor}

\section{Examples} \label{sec:apps}

In this section we briefly examine one of our motivating examples, the irrational rotation algebra, in light of the results of Theorem~\ref{thm:main}. We then apply the work of the previous three sections to compute three different Cartan subalgebras and the associated Weyl groupoids and twists in the case of a specific group $C^*$-algebra. Indeed, many of the simplifications that occur in this setting are due to the fact that the 2-cocycle $c$ is trivial on the subgroups in question, and also on the coset representatives $R(S).$ 
\begin{example}
\label{ex:IRR}

As mentioned in the introduction, the irrational rotation algebra $A_\theta$ can be realized as the twisted group $C^*$-algebra $C_r^*(\mathbb{Z}^2, c_\theta)$ with the 2-cocycle $c_\theta\colon  \mathbb{Z}^2 \times \mathbb{Z}^2 \rightarrow \mathbb{T}$ given by \[c_\theta((n_1, n_2), (m_1, m_2)) = e^{2\pi i \theta n_2m_1}.\]
Notice that if we take $\S \cong \mathbb{Z}$ to be any of the subgroups $  \mathbb{Z} \times \{0\} \cong \{0\} \times \mathbb{Z} \cong \{(n,n) | n \in \mathbb{Z}\}$,  the hypotheses of Theorem~\ref{thm:main} are satisfied. In particular, $c_\theta|_{\S}$ is trivial 
using either of the first two choices for $\S$, so we see that $C_r^*(\mathbb{Z}, c) \cong C_r^*(\mathbb{Z}) \cong C(\mathbb{T})$ is Cartan in $C_r^*(\mathbb{Z}^2, c) \cong A_\theta$. 
Furthermore, a straightforward computation reveals that the action of $\mathbb{Z}^2/\S \cong \mathbb{Z}$ 
on $\mathbb{T}$ is given by  
\[ n \cdot z = e^{2\pi i \theta n} z ,\] 
so that the Weyl groupoid is the usual transformation groupoid $\mathbb{T} \rtimes_\theta \mathbb{Z}.$
Moreover, for all of the subgroups $\S$ identified above, coset representatives $R(\S)$ can be chosen so that $\delta_{r_{[gh]}}^* \delta_{r_{[g]}} \delta_{r_{[h]}} = \delta_e$.  Consequently, the twist $\Sigma$ on the Weyl groupoid is trivial, and Theorem \ref{thm:general-untwist-case} gives us the standard isomorphism $C^*_r(\mathbb{Z}^2, c_\theta) \cong C^*_r(\mathbb{T} \rtimes_\theta \mathbb{Z}).$
\end{example}

{We next apply the results of Section \ref{sec:anna} to a group and 2-cocycle that arose in \cite[Example 8.8]{packer-proj-rep}.  We identify three subgroups $S_0, S_1, S_2$ which give rise to different Cartan subalgebras and different Weyl groupoids.  Indeed, it turns out that the Weyl twist associated to $S_0$ is trivial, but this is not the case for $S_1$ or $S_2$.

While the existence of multiple non-isomorphic Cartan subalgebras in a given $C^*$-algebra is not an uncommon occurrence (cf.~\cite{deeley-putnam-strung}), the novelty of the examples highlighted in this section is the dynamical origin of these subalgebras, arising as they do from subgroups of the initial group.
}

    {For the remainder of this section, l}et $G$ be the discrete group consisting, as a set, of the cartesian product $\mathbb Z^5$ with the following group operation: 
    \begin{align*}
        \mathbf a \cdot \mathbf{b}
        =&\,
        (a_1, a_2, a_3, a_4, a_5)\cdot(b_1,b_2,b_3,b_4,b_5)
        \\:=&\,
        (a_1 + b_1 + 2a_5b_3, a_2+b_2+2a_5b_4, a_3+b_3, a_4+b_4, a_5+b_5)
        .
    \end{align*}
    The inverse of $\mathbf{a}$ is then given by
    \[
        \mathbf{a}\inv = (2a_3a_5-a_1, 2a_4a_5-a_2, -a_3, -a_4, -a_5)
        .
    \]
    Equip $G$ with the following 2-cocycle:
    \[
        c(\mathbf{a},\mathbf{b}) = (-1)^{a_4b_1}
        .
    \]
   The following are three examples of subgroups of $\Iso(G)=G$  each falling into the scope of Theorem \ref{thm:main}:
    \begin{align*}
        S_{0} := \mathbb{Z} \times  \mathbb{Z} \times \{0\}\times  \{0\} \times  \mathbb{Z} 
        , \\
        S_{1} := \mathbb{Z} \times  \mathbb{Z} \times  \mathbb{Z} \times  2\mathbb{Z} \times  \{0\}
        ,
        \\
        S_{2} := 2 \mathbb{Z} \times  \mathbb{Z} \times  \mathbb{Z} \times  \mathbb{Z} \times  \{0\}
        .
    \end{align*}
    \begin{enumerate}
        \item They are subgroups:
        \begin{alignat*}{3}
            &S_{0}:\quad&&
                (a_1,a_2,0,0,a_5)\cdot(b_1,b_2,0,0,b_5)\\
            &   && \qquad = (a_1+b_1, a_2+b_2, 0, 0, a_5+b_5)
                \in S_{0}
                \\
            &   &&
                (a_1,a_2,0,0,a_5)\inv
                =
                (-a_1, -a_2, 0, 0, -a_5)
                \in S_{0}
                \\
            &S_{1}:\quad&&
                (a_1,a_2,a_3,2a_4,0)\cdot(b_1,b_2,b_3,2b_4,0)\\
            &   && \qquad  =
                (a_1+b_1, a_2+b_2, a_3+b_3, 2(a_4+b_4), 0)
                \in S_{1}
                \\
            &   &&
                (a_1,a_2,a_3,2a_4,0)\inv
                =
                (-a_1, -a_2, -a_3, -2a_4, 0)
                \in S_{1}
            \\
            &S_{2}:\quad&&
               (2a_1,a_2,a_3,a_4,0)\cdot(2b_1,b_2,b_3,b_4,0)\\
            &   && \qquad
                =
                (2(a_1+b_1), a_2 + b_2, a_3 + b_3, a_4 + b_4, 0)
                \in S_{2}
                \\
            &   &&
                (2a_1,a_2,a_3,a_4,0)\inv
                =
                (-2a_1, -a_2, -a_3, -a_4, 0)
                \in S_{2}
        \end{alignat*}
           We also immediately see from this that all three are abelian.
        \item The 2-cocycle is trivial on them (not just symmetric):  For $S_0$, it follows from the fourth coordinate being zero. For $S_1$ resp.\ $S_2$, this follows from the evenness of the fourth resp.\ first coordinate.
        \item They are maximal among abelian subgroupoids on which $c$ is symmetric: For $S_0$, we note that, since the last component is all of $\mathbb{Z}$, we need the third and fourth component to be zero for the subgroup to be abelian (which immediately forces the cocycle to be trivial). For $S_{1}$ resp.\ $S_{2}$, allowing an odd number in the fourth \emph{and} first component would make the cocycle non-symmetric, and allowing the last component to be non-trivial would make the subgroups non-abelian. 
        \item They are normal: For $\mathbf{g}, \mathbf{s}\in G$, we have
        \begin{align}
            &\mathbf g\inv \mathbf s\mathbf g = (2g_3g_5-g_1, 2g_4g_5-g_2, -g_3, -g_4, -g_5)
            \notag\\
            &\hphantom{\mathbf g\inv \mathbf s\mathbf g =} \quad  \cdot( s_1 + g_1 + 2s_5g_3, s_2 + g_2 + 2s_5g_4, s_3 + g_3, s_4 + g_4, s_5 + g_5)
            \notag\\
            & = (s_1 + 2s_5g_3 - 2g_5s_3, s_2 + 2s_5g_4 - 2g_5s_4, s_3, s_4, s_5)
            \notag\\
            \begin{split}\label{eq:g-inv-s-g-computed}
             & = {\left\{\begin{array}{ll}
     (s_1 + 2s_5g_3, s_2 + 2s_5g_4, 0, 0, s_5) \in S_0, & \text{ if } \mathbf s \in S_0 \\
     (s_1 - 2g_5s_3, s_2 - 2g_5s_4, s_3, s_4, 0)\in S_i, &  \text{ if } \mathbf s \in S_i\, (i = 1,2). \\
\end{array}\right.}\end{split}
        \end{align}
      
        \item They are immediately centralizing, because $G$ has the unique root property (cf.\ Remark \ref{rmk:Examples-for-(A)}): For any element $\mathbf{g}=(g_1, g_2, g_3, g_4, g_5) \in G$ and $j$ a positive integer, we have
        \[
            \mathbf{g}^{ j}
            =
            (jg_1 +j(j-1)g_5g_3, jg_2+j(j-1)g_5g_4, jg_3, jg_4, jg_5)
            .
        \]
        In particular, we see that $\mathbf{g}^{ j} = \mathbf{h}^{ j}$ implies $\mathbf{g}=\mathbf{h}$.
        \item Since $G$ has the discrete topology, all three are clopen.
    \end{enumerate}
    
   Note that $S_{0}\cong \mathbb{Z}^3$ while  $S_{1}\cong \mathbb{Z}^4 \cong S_2$, so in particular, the Cartan algebras they generate are
    \[
        C_r^*(S_{0}) \cong C(\mathbb{T}^3) \ncong C(\mathbb{T}^4) \cong C_r^* (S_{1}) \cong C^*_r(S_2).
    \]

    Next, we use the machinery developed in Section \ref{sec:anna} to identify the  Weyl groupoids and twists $(\H_i, \Sigma_{i})$ that give rise to
    \[
        \bigl(
            C_r^*(\H_i, \Sigma_{i}), 
            C_0 (\H_i\z)
        \bigr)
        \cong 
        \bigl(
            C_r^*(G, c), 
            C_r^* (S_i)
        \bigr)
        .
    \]


{We begin with $S_{0}$. Since  $c|_{S_0}$ is trivial,  we are in the setting of Corollary \ref{cor:general-Sigma-c-trivial-on-S}. The following Proposition describes the action ${\tilde{\alpha}^{0}}$ on $\H_{0}$ explicitly, and shows that the 2-cocycle $\tilde \sigma$ is trivial in this case.}
\begin{prop}\label{prop:Weyl-H3}
    Let $\H_0 =  \mathbb{Z}^2 \tensor[_{{\tilde{\alpha}^{0}}}]{\ltimes}{}\mathbb{T}^3$, 
    where for $(c,d)\in\mathbb{Z}^2$, $${\tilde{\alpha}^{0}}_{(c,d)} (z_1, z_2, z_3) =
        ( (-1)^d z_1, z_2, z_1^{2c} z_2^{2d} z_3).$$  Then $\H_0$ is the Weyl groupoid associated to the Cartan pair $(C_r^*(G, c), C_r^*(S_0))$, and $(C^*_r(\H_0), C_0(\H_0\z)) \cong (C_r^*(G, c), C^*_r(S_0)).$  That is, the twist associated to this Cartan pair is trivial. 
\end{prop} 

{
\begin{proof}
    For $\mathbf g = (g_1, g_2, g_3, g_4, g_5) \in G,$ let $r_{[\mathbf g]} = (0, 0, g_3, g_4, 0)$ be the representative of $[\mathbf g] \in G/S_0$.  Observe that if $\mathbf s = (s_1, s_2, 0, 0, s_5) \in S_0$ we have 
    \begin{align*} 
    \overline{c(r_{[\mathbf g]}, r_{[\mathbf g^{-1}]})} &= c(r_{[\mathbf g^{-1} \mathbf s]}, r_{[\mathbf g]}) = 1, \qquad c(r_{[\mathbf g^{-1}]}, \mathbf s) = (-1)^{g_4 s_1} , \qquad \text{ and } \\
        r_{[\mathbf g^{-1}]}  \mathbf s r_{[\mathbf g]}&  = (s_1 + 2 s_5 g_3, s_2 + 2s_5 g_4, 0, 0, s_5).
    \end{align*}
    Moreover, $\widehat{S_0} \cong \mathbb{T}^3$, with the pairing of $\widehat{S_0}$ with $\mathbb{T}^3$  given by 
    \[ \langle (s_1, s_2, 0, 0, s_5) , (z_1, z_2, z_3) \rangle_0 := z_1^{s_1} z_2^{s_2} z_3^{s_5}.\]
    Proposition \ref{lem:c_trivial_H} therefore tells us that the action $\tilde{\alpha}^{0}$ of $G/S_0 \cong \mathbb{Z}^2$ on $\mathbb{T}^3$ which gives rise to the Weyl groupoid is 
    \begin{align*} {\tilde{\alpha}^{0}}_{[g_3, g_4]}(z_1, z_2, z_3) & = \left[ (s_1, s_2, 0, 0, s_5) \mapsto (-1)^{g_4 s_1} z_1^{s_1+ 2s_5 g_3} z_2^{s_2+ 2s_5 g_4} z_3^{s_5} \right] \\
    & = \left[ (s_1, s_2, 0, 0, s_5) \mapsto  \left( (-1)^{g_4} z_1\right)^{s_1} z_2^{s_2} \left(z_1^{2g_3} z_2^{2g_4} z_3\right)^{s_5} \right],\end{align*}
    which is the formula for ${\tilde{\alpha}^{0}}$ asserted in the statement of the proposition.
    
   The fact that the Weyl twist is trivial follows from the observation that the set of coset representatives, $R(S_0) = 0 \times 0 \times \mathbb{Z} \times \mathbb{Z} \times 0 $, is also a subgroup of $G$ on which $c$ is trivial, and hence $\delta_{r_{[\mathbf{gh}]}}^* \delta_{r_{[\mathbf g]}} \delta_{r_{[\mathbf h]}} = \delta_{\mathbf{e}}.$ Consequently, the abstract formula for the twist $\tilde{\sigma}$ giving rise to the Weyl twist (see Equation~\eqref{eq:def:tilde-sigma}) is given by
   \[
   \tilde \sigma
   \left(
        ( [\mathbf g], {\tilde{\alpha}^{0}}_{\mathbf h}(\vec z)), ([\mathbf h],\vec{z}) 
    \right)
   = \delta_{\mathbf{e}}(\Psi(\vec{z}))
   = 1
   \]
   for all $\vec{z} \in \mathbb{T}^{3}$ and $\mathbf g, \mathbf h \in G.$  Theorem \ref{thm:general-untwist-case} completes the proof.
\end{proof}}

To compute the Weyl groupoids and twists associated to $S_1$ and $S_2$, we will make use of the function $f: \mathbb Z \to \{ 0,1\}$ given by 
\[f(2k+1) = 1, \ f(2k) = 0.\]
The next proposition deals with the $S_1$ case.

\begin{prop}
The Weyl groupoid associated to the Cartan pair $(C^*_r(G, c), C^*_r(S_1))$ is 
\[
\H_1
= 
 (\under{{S_1}}{G}) \tensor[_{\tilde{\alpha}^1}]{\ltimes}{}\widehat{S}_1  = \left( \mathbb{Z}/2\mathbb{Z} \times \mathbb{Z} \right)\tensor[_{\tilde{\alpha}^1}]{\ltimes}{} \mathbb{T}^4 ,
\]
where the action $\tilde{\alpha}^1$ of $\under{{S_1}}{G}$ on $\mathbb{T}^4$ is given by
\[\tilde{\alpha}^1_{[\bf g]}(z_1, z_2, z_3, z_4) = ((-1)^{g_4} z_1, z_2, z_1^{-2g_5} z_3, z_2^{-4g_5}z_4).\]
The associated Weyl twist is $\H_1 \times_{\tilde \sigma_1 } \mathbb{T},$ where the 2-cocycle $\tilde \sigma_1$ on $\H_1$ satisfies 
\begin{align}
    &\tilde \sigma_1\left(
        (( f (g_4), g_5), \tilde{\alpha}^1_{[\mathbf h]}(\vec{z})), ( ( f (h_4),  h_5), \vec{z})
    \right) =  \frac{ \left( \delta_{r_{[\mathbf g\mathbf h]}}^* \delta_{r_{[\mathbf g]}} \delta_{r_{[\mathbf h]}} \right)(\Psi(\vec{z}))}{\left| \delta_{r_{[\mathbf g\mathbf h]}}^* \delta_{r_{[\mathbf g]}} \delta_{r_{[\mathbf h]}} (\Psi(\vec{z}))\right|}
    \notag 
    \\
    &\qquad = \begin{cases}
    1, &   f (h_4) = 0, \\
    z_2^{2g_5}, &  f (g_4) = 0 \text{ and }  f (h_4) = 1, \\
    z_2^{-2g_5-4h_5}z_4, &  f (g_4) =  f (h_4) = 1.
    \end{cases}
    \label{eq:sigma_1}
\end{align}
\end{prop}
\begin{proof}
    Again, $c|_{S_1 \times S_1}$ is trivial, so we invoke Proposition \ref{lem:c_trivial_H}. {To obtain a concrete formula for the action $\tilde{\alpha}^1$ described in that Proposition,} we first compute
\begin{align*}
    &\overline{c(\mathbf g, \mathbf g\inv)} \, c(\mathbf g\inv, \mathbf s) \, c(\mathbf g\inv \mathbf s, \mathbf g)  = (-1)^{g_4(2g_3g_5 - g_1)}(-1)^{-g_4s_1}(-1)^{(s_4 - g_4)g_1}\\
    &\qquad = (-1)^{g_1s_4+g_4s_1}\\
    &\qquad = \begin{array}{ll}
    (-1)^{g_4s_1} = c(\mathbf g,\mathbf s), & \text{ if } \mathbf s \in S_1 = \mathbb{Z} \times  \mathbb{Z} \times  \mathbb{Z} \times  2\mathbb{Z} \times  \{0\}
    \end{array}
\end{align*}
The fact that the fourth component of $S_1$ is $2\mathbb{Z}$ means that the pairing $\langle \cdot, \cdot \rangle_1$ between $S_1$ and its Pontryagin dual $\widehat{S_1} \cong \mathbb{T}^4$ is given by 
\[  \langle  (s_1, s_2, s_3, s_4, 0), (z_1, z_2, z_3, z_4) \rangle_1 = z_1^{s_1} z_2^{s_2} z_3^{s_3} z_4^{s_4/2}.\]
Thus, evaluating $\tilde{\alpha}^1_{[\mathbf g]}$ at $\nu=(z_1, z_2, z_3, z_4) \in \mathbb{T}^4$ yields
\begin{align*}
    &\tilde{\alpha}^1_{[\mathbf g]}(z_1, z_2, z_3, z_4) = \left[ (s_1, s_2, s_3, s_4,0) \mapsto (-1)^{g_4 s_1} z_1^{s_1 - 2g_5 s_3} z_2^{s_2 - 2g_5 s_4} z_3^{s_3} z_4^{s_4/2} \right] \\&\qquad= \left[ (s_1, s_2, s_3, s_4,0) \mapsto \left((-1)^{g_4} z_1\right)^{s_1}  z_2^{s_2} \left( z_1^{-2g_5} z_3\right)^{s_3} \left(z_2^{ - 4g_5} z_4 \right)^{ s_4/2}  \right],
\end{align*}
i.e.\
\begin{equation}
    \tilde{\alpha}^1_{[\mathbf g]}(z_1, z_2, z_3, z_4) = ((-1)^{g_4} z_1,  z_2,  z_1^{-2g_5}z_3,z_2^{ - 4g_5}z_4 ).
    \label{eq:S_1-action}
\end{equation}
Observe that $\tilde{\alpha}^1$ is topologically free: If a point $\vec z \in \mathbb{T}^4$  is fixed by $\tilde{\alpha}^1_{[\mathbf g]}$ for $\mathbf g \not \in S_1$, we must have $g_4$  even and $z_1, z_2 $ roots of unity.  But then, in any neighborhood of $\vec z$  there are points $\vec w$ for which $w_1, w_2$ are not roots of unity, so there is no neighborhood of $\vec z$ which is fixed by $\tilde{\alpha}^1_{[\mathbf g]}$.  Thus, Theorem \ref{lem:class-of-g-determines-equality} and Proposition \ref{lem:c_trivial_H} tell us that the Weyl groupoid is given by $\H_1 =   (\under{{S_1}}{G})\tensor[_{\tilde{\alpha}^1} ]{\ltimes}{}\widehat{S_1}.$ 

To compute a concrete formula for the 2-cocycle $\tilde \sigma_1$ associated to $S_1$ out of the abstract formula in Equation~\eqref{eq:def:tilde-sigma}, let $R(S_1) = 0 \times 0 \times 0 \times \{ 0,1\} \times \mathbb{Z}$ be our preferred representatives in $G$ of the elements in $\under{ {S_{1}} }{ G }$.  Although $c|_{R(S_1) \times R(S_1)}$ is trivial,  $R(S_1)$ is not a subgroup of $G$.  Indeed, 
\[ r_{[\mathbf g]}^{-1} = (0, 2 f (g_4) g_5, 0, - f (g_4), -g_5)\,
\notin \, R(S_{1}).\]
One then computes that  $1 = c(r_{[\mathbf g \mathbf h]}, r_{[\mathbf g \mathbf h]}^{-1}) = c(r_{[\mathbf g \mathbf h]}^{-1}, r_{[\mathbf g]})  = c(r_{[\mathbf g]},r_{[\mathbf h]})$ and thus that \begin{align*}
    \delta_{r_{[\mathbf g \mathbf h]}}^* \delta_{r_{[\mathbf g]}} \delta_{r_{[\mathbf h]}}
    &=
    \delta_{r_{[\mathbf g \mathbf h]}\inv r_{[\mathbf g]} r_{[\mathbf h]}}.
\end{align*}
Explicitly, we have 
\begin{align*}
    r_{[\mathbf g \mathbf h]}\inv r_{[\mathbf g]} r_{[\mathbf h]}
     =&\,
    \left(0, 2g_5  f (h_4) + 2(g_5+h_5) f (g_4+h_4)- 2(g_5 +h_5)( f (g_4) +  f (h_4)), 
    \right.
    \\
    &\quad \left.
    0,
    - f (g_4+h_4)+  f (g_4) +  f (h_4), 0
    \right).
\end{align*}
If $f(h_4) = 0$, so that $f(g_{4}+h_{4})=f(g_{4})$, then we have $r_{[\mathbf g \mathbf h]}\inv r_{[\mathbf g]} r_{[\mathbf h]} = \mathbf{e}.$ However, if $f(h_4) = 1$, then
\begin{align*}
r_{[\mathbf g \mathbf h]}\inv r_{[\mathbf g]} r_{[\mathbf h]} &=
\begin{cases}
(0, 2g_5, 0, 0, 0), &\text{ if } f(g_4) = 0 ,\text{ and }
\\
(0, -2g_5 - 4 h_5, 0, 2, 0),&\text{ if } f(g_4) = 1.
\end{cases}
\end{align*}
Thus, the 2-cocycle $\tilde \sigma_1$ on $\H_1$ is given by Equation \eqref{eq:sigma_1}, as claimed.
\end{proof}

One might suspect that the symmetry between $S_1$ and $S_2$ would result in the associated Weyl groupoids and twists being isomorphic.  This is not the case, as we now show.
\begin{prop}
The Weyl groupoid associated to the Cartan pair $(C^*_r(G, c), C^*_r(S_2))$ is 
\[
    \H_2 =
    (\under{{S_2}}{G}) \tensor[_{\tilde{\alpha}^2}]{\ltimes}{} \widehat{S}_2
    =
    \left( \mathbb{Z}/2\mathbb{Z} \times \mathbb{Z}\right) \tensor[_{\tilde{\alpha}^2}]{\ltimes}{} \mathbb{T}^4
,\]
where the action $\tilde{\alpha}^2$ of $\under{{S_2}}{G} \cong \mathbb{Z}/2\mathbb{Z} \times \mathbb{Z}$ on $\mathbb{T}^4$ is given by \[\tilde{\alpha}^2_{[\bf g]}(z_1, z_2, z_3, z_4) = ( z_1, z_2, z_1^{-g_5} z_3, (-1)^{g_1}z_2^{-2g_5}z_4).\]
The associated Weyl twist is $\H_2 \times_{\tilde \sigma_2 } \mathbb{T},$ where the 2-cocycle $\tilde \sigma_2$ on $\H_2$ satisfies 
\begin{align}
    &\tilde \sigma_2
    \left(
        ( ( f (g_1), g_5), \tilde{\alpha}^2_{[\mathbf h]}(\vec{z})), ( ( f (h_1),  h_5),\vec{z})
    \right) =   \begin{cases}
    z_1, &   f (g_1) = f(h_1) = 1, \\
  1, &  \text{else.}
    \end{cases}
    \label{eq:sigma_2}
\end{align}
\end{prop}
\begin{proof} 
Again, we begin by computing for $\mathbf g \in G$ and $\mathbf s \in S_2 = 2\mathbb{Z} \times \mathbb{Z} \times \mathbb{Z} \times  \times \mathbb{Z} \times \{ 0 \} $:
\begin{align*}
    \overline{c(\mathbf g, \mathbf g\inv)} \, c(\mathbf g\inv, \mathbf s) \, c(\mathbf g\inv \mathbf s, \mathbf g) & = (-1)^{g_4(2g_3g_5 - g_1)}(-1)^{-g_4s_1}(-1)^{(s_4 - g_4)g_1}\\
    &=   (-1)^{g_1s_4} = c(\mathbf s,\mathbf g).
\end{align*}
The pairing $\langle \cdot, \cdot \rangle_2$ between $S_2$ and $\mathbb{T}^4$ is given by 
\[ \langle ( s_1, s_2, s_3, s_4, 0), (z_1, z_2, z_3, z_4) \rangle_2 = z_1^{s_1/2} z_2^{s_2} z_3^{s_3} z_4^{s_4},\]
and so the fact that $\mathbf{g}^{-1} \mathbf{s} \mathbf{g} = (s_1 - 2g_5 s_3, s_2 - 2g_5 s_4, s_3, s_4, 0)$ implies 
\begin{align*}
    &\tilde{\alpha}^2_{[\mathbf g]}(z_1, z_2, z_3, z_4) = \left[ (s_1, s_2, s_3, s_4,0) \mapsto (-1)^{g_1 s_4} z_1^{s_1/2 - g_5 s_3} z_2^{s_2 - 2g_5 s_4} z_3^{s_3} z_4^{s_4}\right] \\&\qquad= \left[ (s_1, s_2, s_3, s_4,0) \mapsto z_1^{s_1/2} z_2^{s_2} \left( z_1^{-g_5} z_3\right)^{s_3} \left( (-1)^{g_1} z_2^{-2g_5} z_4\right)^{s_4} \right].
\end{align*} 
Consequently, 
\begin{equation}
\tilde{\alpha}^2_{[\mathbf g]}(z_1, z_2, z_3, z_4) = (z_1, z_2, z_1^{-g_5} z_3, (-1)^{g_1} z_2^{-2g_5} z_4).
    \label{eq:S_2-action}
\end{equation}
Again, the action $\tilde{\alpha}^2$ is topologically free, because if $\tilde{\alpha}^2_{[\mathbf g]}(\vec z) = \vec z$ and $\mathbf g \not \in S_2$,  $z_1$ and $z_2$ must be roots of unity.  Since every neighborhood of such a point contains points $\vec w$ with $w_1 \not=1$, $\tilde{\alpha}^2$ is topologically free.

We now compute $\tilde \sigma_2$, using $R(S_2) = \{ 0 ,1 \} \times 0 \times 0 \times 0 \times \mathbb{Z},$ so that 
\[ r_{[\mathbf g]} = ( f (g_1), 0, 0, 0, g_5) \qquad \text{ and } \qquad r_{[\mathbf g]}^{-1} = (- f (g_1), 0, 0, 0, -g_5) .\]
{Note that the latter might not be an element of $R(S_{2})$.}
As in the $S_1$ case, $1 = c(r_{[\mathbf g \mathbf h]}, r_{[\mathbf g \mathbf h]}^{-1}) = c(r_{[\mathbf g \mathbf h]}^{-1}, r_{[\mathbf g]}) =c(r_{[\mathbf g]}, r_{[\mathbf h]})$, and so
\begin{align}
    \tilde{\sigma_2}
    \bigl( 
        (  [\mathbf g],  \tilde{\alpha}^2_{[\mathbf h]}(\vec z)  ) , ([\mathbf h], \vec z)
    \bigr)
  &= \delta_{r_{[\mathbf g \mathbf h]}^{-1} r_{[\mathbf g]} r_{[\mathbf h]}}(\Psi(\vec z)) \notag \\
  &= \delta_{( f (g_1) +  f (h_1) -  f (g_1+h_1), 0, 0, 0, 0)}(\Psi(\vec z)) \notag \\
  &= \begin{cases}
  z_1, &  f (g_1) =  f (h_1) = 1, \\
  1, & \text{ else.}
  \end{cases} \notag
\end{align}
{This concludes our proof.}
\end{proof}

\section{The necessity of being immediately centralizing}\label{ex:masa-counterexample} 
    The following pathological group was constructed to establish the necessity of the  ``immediately centralizing'' hypothesis in Theorem \ref{thm:main}.

	Let $ G $ be the set $\mathbb{Z}/4\mathbb{Z} \times \mathbb{Z}/4\mathbb{Z} \times \mathbb{Z} \times \mathbb{Z} \times \mathbb{Z} /4\mathbb{Z} $, and define multiplication on $ G $ by
		\begin{multline*}
		([a]_4,[b]_4,c,d,[e]_4) \cdot ([a']_4,[b']_4,c',d',[e']_4)
		\\
		= ([a+a'+2ec']_4,[b+b'+2ed']_4,c+c',d+d',[e+e']_4).
		\end{multline*}
		One can check that $ G $ is a group with inverse given by
		\begin{equation*}
			([a]_4,[b]_4,c,d,[e]_4)\inv = ([2ec-a]_4,[2ed-b]_4,-c,-d,[-e]_4).
		\end{equation*}
		
		Define $c\colon   G  \times  G  \to \mathbb{T}$ by 
		\[
			c(([a]_4 ,[b]_4 ,c,d,[e]_4 ), ([a']_4 ,[b']_4 ,c',d',[e']_4 )) = (-1)^{da'}.
		\]
		Then $c$ is a $2$-cocycle on $ G $ because $c(\mathbf a, \mathbf 0) = c(\mathbf 0, \mathbf a) = 1$ for all $\mathbf a \in  G $ and 
		\begin{equation*}
			c(\mathbf a_1, \mathbf a_2)
			c(\mathbf a_1 \mathbf a_2, \mathbf a_3) = (-1)^{d_1 a_2} (-1)^{(d_1+d_2)a_3} 
		\end{equation*}
		which equals
		\begin{equation*}
			c(\mathbf a_1, \mathbf a_2 \mathbf a_3) c(\mathbf a_2, \mathbf a_3)
			=
			(-1)^{d_1 (a_2+a_3+2e_2c_3)} (-1)^{d_2 a_3}
			=
			(-1)^{d_1a_2 + d_1 a_3} (-1)^{d_2 a_3}.
		\end{equation*}
		
		Now define $S=\mathbb{Z}/4\mathbb{Z} \times \mathbb{Z}/4\mathbb{Z} \times \mathbb{Z} \times 2\mathbb{Z} \times \{[0]_4,[2]_4\}$.
		
		\begin{prop}
			$S$ is a subgroup that is maximal among abelian subgroups on which $c$ is symmetric.
		\end{prop}
		
		\begin{proof}
			Suppose $([a]_4 ,[b]_4 ,c,d,[e]_4 ), ([a']_4 ,[b']_4 ,c',d',[e']_4 ) \in S$. Then
			\begin{align*}
			&
			([a]_4 ,[b]_4 ,c,d,[e]_4 ) \cdot ([a']_4 ,[b']_4 ,c',d',[e']_4 )
			\\
			&\qquad = ([a+a'+2ec']_4,[b+b'+2ed']_4,c+c',d+d',[e+e']_4)
			\\
			&\qquad = ([a+a']_4,[b+b']_4, c+c', d+d', [e+e']_4),
			\end{align*}
			since $e$ is even so $[2ec']_4=[2ed']_4=0$. Similarly,
			\begin{align*}
			&([a']_4 ,[b']_4 ,c',d',[e']_4 ) \cdot ([a]_4 ,[b]_4 ,c,d,[e]_4 )
			\\
			&\qquad = ([a'+a+2e'c]_4,[b'+b+2e'd]_4,c'+c,d'+d,[e'+e]_4)
			\\
			&\qquad = ([a'+a]_4, [b'+b]_4, c'+c, d'+d, [e'+e]_4).
			\end{align*}

			The 2-cocycle is trivial on $S$ because the fourth component is even. 
			
			Regarding maximality, assume $T$ is a subgroup of $ G $ that contains $S$. If $T$ contains an element whose fourth component is odd, then $c$ is not symmetric on $T$. If $T$ contains an element whose fifth component is odd, then $T$ is not abelian. Thus $S$ is maximal among abelian subgroups of $ G $ on which $c$ is symmetric.
		\end{proof}

		\begin{prop}
			$G$ and $S$ satisfy all other assumptions of Theorem \ref{thm:main}, except that $S$ is not immediately centralizing. 
		\end{prop}
		
		\begin{proof}
		    Since $G$ is a countable group with the discrete topology, $G$ is a second countable locally compact Hausdorff \'etale groupoid and $S$ is open and closed. To see that $S$ is normal, we must check that $\mathbf{a}\inv \mathbf{s}\mathbf{a} \in S$ for all $\mathbf{a} \in G$ and $\mathbf{s} \in S$.  Note that the fourth component of $\mathbf{a}\inv \mathbf{s}\mathbf{a}$ equals the fourth component of $\mathbf{s}$, and the fifth component of $\mathbf{a}\inv \mathbf{s}\mathbf{a}$ equals the fifth component of $\mathbf{s}$.  It therefore follows that if $\mathbf s \in S, \mathbf a \in G$, then $\mathbf{a}\inv \mathbf{s}\mathbf{a}\in S$. Thus $S$ is normal.
		    
		    {To see that $S$ is not immediately centralizing, consider $\mathbf g = ([0]_4, [0]_4, 0, 0, [1]_4).$  Observe that if $\mathbf s = ([a]_4, [b]_4, c, 2d, [e]_4)\in S$ then 
		    \[\mathbf g \mathbf s \mathbf g^{-1} = ([a +2c]_4, [b]_4, c, 2d, [e]_4),\]
		    while $\mathbf g \mathbf s^2 \mathbf g^{-1} = ([2a]_4, [2b]_4, 2c, 4d, [2e]_4) = \mathbf s^2. $
		    Since there are elements $\mathbf s \in S$ for which $([a +2c]_4, [b]_4, c, 2d, [e]_4) \not= ([a]_4, [b]_4, c, 2d, [e]_4), $ we see that $\mathbf g$ is 2-centralizing but not 1-centralizing.}
		\end{proof}
		
		\begin{prop}
		    $C_r^*(S)$ is not maximal abelian.
		\end{prop}
		
		\begin{proof}
		    Define the function $h\colon  G  \to \mathbb{C}$ by $h = \delta_{\mathbf{\nu}} + \delta_{\mathbf{\mu}}$
		    where $\mathbf{\nu} := (0,0,0,0,[1]_4)$ and $\mathbf\mu := ([2]_4,0,0,0,[1]_4)$. Then $h$ does not have support in $S$, so $h\notin C_{r}^{*}(S,c)$, but we will prove that it commutes with every function in $ C_c(S)$.
		    
			Suppose $\varphi \in C_c(S)$ and $\mathbf{a}=([a]_4,[b]_4,c,d,[e]_4) \in  G $. Then
			\begin{multline*}
			h *  \varphi (\mathbf{a}) = \sum_{\mathbf\alpha \mathbf\beta = \mathbf{a}} h(\mathbf\alpha)\varphi(\mathbf\beta) c(\mathbf\alpha, \mathbf\beta) = \varphi(\mathbf\nu\inv \mathbf{a})c(\mathbf\nu,\mathbf\nu\inv \mathbf{a}) + \varphi(\mathbf\mu\inv \mathbf a)c(\mathbf\mu, \mathbf\mu\inv \mathbf a).
			\end{multline*}
			Notice that $c(\mathbf\nu,\mathbf\nu\inv t) = c(\mathbf\mu,\mathbf\mu\inv \mathbf a) =1$ since the fourth component of $\mathbf\nu$ and $\mathbf\mu$ is $0$. Thus 
			\begin{equation*}
			h *  \varphi (\mathbf a) = \varphi(\mathbf\nu\inv \mathbf a) + \varphi(\mathbf\mu\inv \mathbf a).
			\end{equation*}
			
			On the other hand, 
			\begin{equation*}
			\varphi*h(\mathbf a) = \sum_{\mathbf\alpha'\mathbf\beta'=\mathbf a} \varphi(\mathbf\alpha') h(\mathbf\beta') c(\mathbf\alpha',\mathbf\beta') = \varphi(\mathbf a \mathbf\nu\inv)c(\mathbf a \mathbf\nu\inv,\mathbf\nu) + \varphi(\mathbf a \mathbf\mu\inv)c(\mathbf a \mathbf\mu\inv,\mathbf\mu).
			\end{equation*}
			Again, $c(\mathbf a \mathbf\nu\inv,\mathbf\nu) = c(\mathbf a \mathbf\mu\inv,\mathbf\mu) = 1$ since the first components of $\mathbf\nu$ and $\mathbf\mu$ are even. Thus 
			\begin{equation*}
			\varphi*h(\mathbf a) = \varphi(\mathbf a \mathbf\nu\inv) + \varphi(\mathbf a \mathbf\mu\inv)
			\end{equation*}
			
			Since $\mathbf\nu\inv = ([0]_4,[0]_4,0,0,[3]_4)$ and $\mathbf\mu\inv = ([2]_4,[0]_4,0,0,[3]_4)$, then 
			\begin{eqnarray*}
			\mathbf\nu\inv \mathbf a &=& ([a+6c]_4, [b+6d]_4,c,d,[e+3]_4) \\
			\mathbf a \mathbf\nu\inv &=& ([a]_4,[b]_4,c,d,[e+3]_4) \\
			\mathbf\mu\inv \mathbf a &=& ([a+2+6c]_4,[b+6d]_4,c,d,[e+3]_4)\\
			\mathbf a \mathbf\mu\inv &=& ([a+2]_4,[b]_4,c,d,[e+3]_4).
			\end{eqnarray*}
			
			Since $\varphi$ is supported on $S$, then $\varphi*h(\mathbf a)$ is nonzero only if $\mathbf a \mathbf\nu\inv$ or $\mathbf a \mathbf\mu\inv$ is in $S$, i.e., if $d \in 2\mathbb{Z}$ and $e \in \{[1]_4,[3]_4\}$. Similarly, $h*\varphi(\mathbf a)$ is nonzero only if $\mathbf\nu\inv \mathbf a$ or $\mathbf\mu\inv \mathbf a$ is in $S$, i.e., if $d \in 2\mathbb{Z}$ and $e \in \{[1]_4,[3]_4\}$. 
			
			Thus if $d \notin 2\mathbb{Z}$ or $e \notin \{[1]_4,[3]_4\}$, then $\varphi * h(\mathbf a) = 0 = h*\varphi(\mathbf a)$. Now let us consider the case when $d \in 2\mathbb{Z}$ and $e \in \{[1]_4,[3]_4\}$). 
			
			If $c=2k$ for some $k \in \mathbb{Z}$, then
			\begin{align*}
			\mathbf\nu\inv \mathbf a &= (([a+6(2k)]_4, [b+6d]_4,2k,d,[e+3]_4) \\&= ([a]_4, [b]_4,2k,d,[e+3]_4) = \mathbf a \mathbf\nu\inv
			\end{align*}
			and
			\begin{align*}
			\mathbf\mu\inv \mathbf a &= ([a+2+6(2k)]_4,[b+6d]_4,2k,d,[e+3]_4) \\&= ([a+2]_4,[b]_4,2k, d,[e+3]_4) = \mathbf a \mathbf\mu\inv
			.
			\end{align*}
			
			On the other hand, if $c=2k+1$ for some $k \in \mathbb{Z}$, then
			\begin{align*}
			\mathbf\nu\inv \mathbf a &= ([a+6(2k+1)]_4, [b+6d]_4,2k+1,d,[e+3]_4) \\&= ([a+2]_4, [b]_4, 2k+1,d,[e+3]_4) = \mathbf a \mathbf\mu\inv
			\end{align*}
			and
			\begin{align*}
			\mathbf\mu\inv \mathbf a &= ([a+2+6(2k+1)]_4,[b+6d]_4,2k+1,d,[e+3]_4) \\&= ([a]_4,[b]_4,2k+1,d,[e+3]_4) = \mathbf a \mathbf\nu\inv
			.
			\end{align*}
			
			Thus, in each case, $\varphi * h(\mathbf a) = h*\varphi(\mathbf a)$ for all $\mathbf a \in  G $ 
		    and all $\varphi\in C_{c}(S)$, and hence all $\varphi\in C_{r}^{*}(S)$.
		\end{proof}

\appendix

\section{Proof of Lemma \ref{lem:eta-in-S}}
\label{app:lemma-4-8}

We still owe the reader the proof of the following lemma:
\newtheorem*{lemma-4-8}{Lemma \ref{lem:eta-in-S}}
\begin{lemma-4-8}
    \emph{Suppose $\G$ is an {\'e}tale groupoid, $c$ is a 2-cocyle on $\G$, and 
	$\S$ is maximal among abelian subgroupoids of $\Iso (\G)$ on which $c$ is symmetric. Let $u$ be a unit. If $\eta\in\G^{u}_{u}$ satisfies $\eta s = s \eta$ and $c(s, \eta) = c(\eta, s)$ for all $s\in \S^{u}_{u}$, then $\eta\in\S$.}
\end{lemma-4-8}

To do so, we require a few smaller results.

	\begin{lemma}\label{lem:eta-or-etainv}
		Suppose $\G$ is a groupoid with a 2-cocycle $c$ and $u$ is  a unit in $\G$. If $\eta,\xi\in\G^{u}_{u}$ commute, then the following statements are equivalent:
		\begin{align}
			\label{eq:c-s-eta}			c(\xi , \eta) &= c(\eta, \xi)\\
			\label{eq:c-seta-etainv}	c(\xi\eta, \eta\inv ) &= c(\eta\inv , \xi\eta)\\
			\label{eq:c-s-etainv}		c(\xi , \eta\inv) &= c(\eta\inv, \xi)\\
			\label{eq:c-setainv-eta}	c(\xi\eta\inv , \eta) &= c(\eta, \xi\eta\inv )
		\end{align}
	\end{lemma}
	
	\begin{proof}
		We will show \mbox{\eqref{eq:c-s-eta}$\implies$\eqref{eq:c-seta-etainv}$\implies$\eqref{eq:c-s-etainv}.} Replacing $\eta$  by $\eta\inv$, the same argument also gives \mbox{\eqref{eq:c-s-etainv}$\implies$\eqref{eq:c-setainv-eta}$\implies$\eqref{eq:c-s-eta}} also, and so all conditions are equivalent.
			
		Assume  \eqref{eq:c-s-eta} holds.  Since $c(u, \xi) = c(\xi, u) = 1$, we use the cocycle condition, and our hypotheses that $\xi \eta = \eta \xi$ and $c(\xi, \eta) = c(\eta, \xi)$, to see that
		\begin{align*}
			c(\xi\eta,\eta\inv)
			&= c(\xi,\eta\eta\inv)\,c(\eta,\eta\inv)\,\overline{c(\xi,\eta)}
		   = c(\eta\inv,\eta)\,\overline{c(\xi,\eta)} 
		   \\
			& = c(\eta\inv\eta,\xi)\,c(\eta\inv,\eta)\,\overline{c(\eta,\xi)}
			\\
			&=c(\eta\inv,\eta\xi)
			=c(\eta\inv,\xi\eta).
		\end{align*}
Thus \eqref{eq:c-s-eta}$\implies$\eqref{eq:c-seta-etainv}.
		
		Next, assume Equation \eqref{eq:c-seta-etainv}, i.e.\@ $c(\xi\eta, \eta\inv ) = c(\eta\inv , \xi\eta)$. We compute
		\begin{align*}
			c(\eta\inv,\xi)
			&=
			c(\eta\inv, (\xi\eta)\eta\inv)
			=
			c(\eta\inv(\xi\eta), \eta\inv)\,c(\eta\inv, \xi\eta)\,\overline{c(\xi\eta,\eta\inv)}
			\\
			&=	c(\eta\inv(\xi\eta), \eta\inv)
			= c(\xi,\eta\inv),
		\end{align*}
		which is exactly Equation \eqref{eq:c-s-etainv}.	This concludes the proof.
	\end{proof}

\begin{lemma}\label{lem:eta-equivalences}
		Suppose $\G$ is a groupoid with a 2-cocycle $c$ and $u$ is a unit in $\G$. Assume further that $\S$ is an abelian subgroupoid of $\G$ on which $c$ is symmetric. If  $\eta\in\G^{u}_{u}$ satisfies $\eta s = s \eta$ for some $s\in \S^{u}_{u}$, the following are equivalent:
		\begin{enumerate}[label=\arabic*., font=\upshape, ref=\arabic*]
			\item\label{hyp:c-s-eta}	$c(s , \eta) = c(\eta, s)$.
			\item\label{hyp:c-s-teta}	For all $t\in \S^{u}_{u}$, we have $c(s , \eta t) = c(\eta t, s)$.
			\item\label{hyp:ex-t}		For some $t\in \S^{u}_{u}$, we have $c(s , \eta t) = c(\eta t, s)$.
		\end{enumerate}
	\end{lemma}
	
\begin{proof}
	Suppose  Hypothesis \ref{hyp:c-s-eta} holds. Then
		\begin{align*}
			c(s,\eta t) &=  c(s\eta,t) \, c(s,\eta) \, \overline{c(\eta, t)}
		    =
			c(\eta s,t) \, c(\eta,s) \, \overline{c(\eta, t)}			
			\\
			&=
			\bigl(
				c(\eta, st) \, c(s, t) \, \overline{c(\eta, s)}
			\bigr)
		 \, c(\eta,s) \, \overline{c(\eta, t)}	
		=
			c(\eta, ts) \, c(t, s) \, \overline{c(\eta, t)}	
			\\
			&=
			c(\eta t,s).
		\end{align*}
		The penultimate equality follows from the fact that $\S$ is abelian and  $c$ is symmetric on $\S$.
		
		To see that Hypothesis \ref{hyp:ex-t} implies Hypothesis \ref{hyp:c-s-eta}, note that the cocycle condition implies that for any $t \in \G^u,$
		\begin{align*}
			c(s,\eta) =&\overline{c(s\eta,t)}  \,  c(s,\eta t)  \,  c(\eta, t), \quad \text{ and}\\
			c(\eta,s) =& \overline{c(\eta s,t)} \, c(\eta, st)\,c( s,t).
	    \end{align*}
		Since $\S$ is abelian and $c$ is symmetric on $\S$, we can rewrite the second equality as follows, using the cocyle condition again for the last step:
		\[	c(\eta,s) = \overline{c(\eta s, t)} c(\eta, ts)\, c(t,s) = \overline{c(\eta s, t)} c(\eta t, s) c(\eta, t).
		\]
	Hypothesis \ref{hyp:ex-t} and the assumption that $s \eta = \eta s$ imply that the right hand sides of these equations agree; that is, Hypothesis \ref{hyp:ex-t} implies Hypothesis \ref{hyp:c-s-eta}.
	\end{proof}

	\begin{proof}[Proof of Lemma~\ref{lem:eta-in-S}]
		First of all, note that Lemma \ref{lem:eta-equivalences} shows that $c(s , \eta t) = c(\eta t, s)$ for any $s, t\in \S^u_u$.  Moreover, the fact that $\eta^{-1}s = s\eta^{-1}$ for any $s \in \S^u_u,$ together with the cocycle condition and Equation \eqref{eq:c-s-etainv}, imply that
		
		\begin{equation}\label{eq:c-s-teta}
			c(t\eta\inv,\eta) = c(\eta,t\eta\inv)
		\end{equation}
		for all $t\in\S^{u}_{u}$.

		Let $\T$ be the subgroupoid of $\Iso(\G)$ generated by $\S$ and $\eta$; note that $\T$ is abelian. We want to show that $c$ is symmetric on $\T$, so that maximality of $\S$ implies $\S=\T$. Since an arbitrary element of $\T$ is either in $\S$ or of the form $s\eta^{k}$ for $s\in\S^{u}_{u}$ and $k\in\mathbb{Z}$, we have to show 
		\begin{equation}\label{eq:c-sym-on-T}
			c(s\eta^{k}, t\eta^{n}) \overset{!}{=} c(t\eta^{n}, s\eta^{k})
		\end{equation}
		for all $s,t\in\S^{u}_{u}$ and all $n,k\in\mathbb{Z}$. 
		
		First, we prove the case $n=1$, $t=u$, and $k\in\mathbb{N}_{0}$: the base case $k=0$ is one of the assumptions.
		We compute
		\begin{alignat*}{3}
			c(\eta,s\eta^{k+1})	
			&= c(\eta (s\eta^{k}),\eta) \, c(\eta,s\eta^{k}) \,\overline{c(s\eta^{k},\eta)}
			&\quad\text{(cocycle condition)}
			\\
			&= c(\eta (s\eta^{k}),\eta) \, c(s\eta^{k},\eta) \,\overline{c(s\eta^{k},\eta)}
			&\quad\text{(induction hypothesis)}
			\\
			&=c(\eta (s\eta^{k}),\eta) = c(s\eta^{k+1},\eta)
			,
		\end{alignat*}
		so we have shown that for all $s\in\S^{u}_{u}$ and $k\in\mathbb{N}_{0}$,
		\begin{equation}\label{eq:n=1-t=u}
			c(s\eta^{k}, \eta)=c(\eta, s\eta^{k}). 
		\end{equation}

		Next, we want to show the case $n=0$ and $k\in\mathbb{N}_{0}$: the base case $k=0$ is true since $c$ is assumed symmetric on $\S$. We compute		
		\begin{align}
			c(s\eta^{k+1},t)
			&= c(\eta,(s\eta^{k})t) \, c(s\eta^{k},t) \, \overline{c(\eta,s\eta^{k})}
			= c(\eta,(ts)\eta^{k}) \, c(s\eta^{k},t) \, \overline{c(\eta,s\eta^{k})}
			\notag
			\\
			&= c((ts)\eta^{k},\eta) \, c(s\eta^{k},t) \, \overline{c(s\eta^{k},\eta)}
			= c(t(s\eta^{k}),\eta) \, c(t,s\eta^{k}) \, \overline{c(s\eta^{k},\eta)}
			\notag
		\\ &	=
			c(t,s\eta^{k+1}). \notag
		\end{align}
		Thus, we have shown for all $s,t\in\S^{u}_{u}$ and $k\in\mathbb{N}_{0}$,
		\begin{equation}\label{eq:n=0}
			c(s\eta^{k}, t)=c(t, s\eta^{k}).
		\end{equation}

		Now, let us show Equation \eqref{eq:c-sym-on-T} for arbitrary $n,k\in\mathbb{N}_{0}$ by induction on $n$. Equation \eqref{eq:n=0}, which holds for all $k\in\mathbb{N}_{0}$, is the base case $n=0$.
		We compute
		\begin{alignat*}{3}
			c(s\eta^{k}, t\eta^{n+1})
			&
			= c(s\eta^{k}(t\eta^{n}), \eta)\,c(s\eta^{k},t\eta^{n})\,\overline{c(t\eta^{n}, \eta)}
			&\quad\text{(cocycle condition)}
			\\
			&= c(\eta, st\eta^{k+n})\,c(s\eta^{k},t\eta^{n})\,\overline{c(\eta,t\eta^{n})}
			&\quad\text{(Eq.~\eqref{eq:n=1-t=u} twice)}
			\\
			&= c(\eta, st\eta^{k+n})\,c(t\eta^{n},s\eta^{k})\,\overline{c(\eta,t\eta^{n})}
			&\quad\text{(induction hypothesis)}
			\\
			&=
			c(\eta,(t\eta^{n})s\eta^{k})\,c(t\eta^{n},s\eta^{k})\,\overline{c(\eta,t\eta^{n})}
			&
			\\
			&=
			c(t\eta^{n+1},s\eta^{k}).
			&\quad\text{(cocycle condition)}
		\end{alignat*}

		To sum up: if $\eta s=s\eta$ and $c(s,\eta)=c(\eta,s)$ for every $s\in\S^{u}_{u}$, then
		\begin{alignat*}{2}
		    \forall \  s,t\in\S^{u}_{u}, \forall \  n,k\in\mathbb{N}_{0},
		    \qquad
			c(s\eta^{k}, t\eta^{n})
			=
			c(t\eta^{n},s\eta^{k}).
		\end{alignat*}
		Since $\eta\inv s=s\eta\inv$ and since  Lemma \ref{lem:eta-or-etainv} implies $c(s,\eta\inv)=c(\eta\inv,s)$, the same proof shows
		\begin{align*}
		    \forall \  s,t\in\S^{u}_{u}, \forall \  n,k\in\mathbb{N}_{0},
		    \qquad
			c(s(\eta\inv)^{k}, t(\eta\inv)^{n})
			=
			c(t(\eta\inv)^{n},s(\eta\inv)^{k}),
		\end{align*}
		or in other words, we have for all $s,t\in\S^{u}_{u}$ and $n,k\in\mathbb{N}_{0}$
		\begin{align}\label{eq:c-sym-on-T-negative-exp}
			c(s\eta^{-k}, t\eta^{-n})
			=
			c(t\eta^{-n},s\eta^{-k})
			.
		\end{align}
		It remains to check that
	$c(s\eta^{k}, t\eta^{-n})
			=
			c(t\eta^{-n},s\eta^{k})
		$ for $n,k\in\mathbb{N}^{\times}$,
		which we do by another induction.	
		We check the base case $n=1$: 
 		\begin{align*}
 			&c(s\eta^{k}, t\eta\inv)
 			=
 			c(s\eta^{k-1}, \eta t\eta\inv)\,c(\eta, t\eta\inv)\,\overline{c(s\eta^{k-1}, \eta)}
 			\\
 			&\quad=
 			c(s\eta^{k-1}, t)\,c(\eta, t\eta\inv)\,\overline{c( \eta,s\eta^{k-1})}
 			=
  			c(s\eta^{k-1}, t)\,c(t\eta\inv, \eta)\,\overline{c( \eta,s\eta^{k-1})}
   			\\
   			&\quad=
   			c(t,s\eta^{k-1})\,c(t\eta\inv, \eta)\,\overline{c( \eta,s\eta^{k-1})}
   		=			
			c(t\eta\inv,s\eta^{k})
			.
 		\end{align*}

		Our final induction hypothesis is:
		\begin{equation*}
		    \text{For a fixed $n\geq 1$ and all $k\in\mathbb{N}_{0}, s,t\in\S^{u}_{u}$,}
		    \quad
			c(s\eta^{k}, t\eta^{-n})
			=
			c(t\eta^{-n},s\eta^{k})
			.
		\end{equation*}
		We compute with the cocycle condition
		\begin{alignat*}{2}
			c(s\eta^{k}, t\eta^{-(n+1)})
			&=
			c((s\eta^{k})\eta\inv,t\eta^{-n})\,c(s\eta^{k},\eta\inv)\,\overline{c(\eta\inv,t\eta^{-n} )}
			&
			\\
			&=
			c(s\eta^{k-1},t\eta^{-n})\,c(s\eta^{k},\eta\inv)\,\overline{c(t\eta^{-n},\eta\inv )}
			&\text{(Eq.~\eqref{eq:c-sym-on-T-negative-exp})}
			\\
			&=
			c(t\eta^{-n},\eta\inv s\eta^{k})\,c(\eta\inv, s\eta^{k})\,\overline{c(t\eta^{-n},\eta\inv )}
            &
			\\
			&=					
			c(t\eta^{-(n+1)},s\eta^{k}).
			&\text{(cocycle condition)}
		\end{alignat*}
		This concludes our proof.
	\end{proof}

\section*{Acknowledgments}
We thank the mathematics departments at Northwestern University and Fitchburg State University for support during the visits of the research group to these institutions. We also thank Aidan Sims for his contributions to Proposition \ref{prop:weyl-picture}.

\bibliographystyle{amsalpha} 
\bibliography{eagbib}


\end{document}